\documentclass[12pt,a4paper]{article}
\usepackage{amsfonts}
\usepackage{amssymb}
\usepackage{amsmath}
\usepackage{amsthm}
\usepackage{latexsym}

\newcommand{\abs}[1]{\lvert#1\rvert}

\newcommand{\grad}[1]{{\nabla} #1} 
\renewcommand{\div}[1]{{\nabla} \cdot #1} 
\newcommand*\Laplace{\mathop{}\!\mathbin\bigtriangleup}
\newcommand{\be}{\begin{equation}}
\newcommand{\ee}{\end{equation}}

\newcommand{\pa}{\partial}
\newcommand{\an}[1]{\langle #1 \rangle}
\newcommand{\RR}{\mathbb R}
\newcommand{\eps}{\varepsilon}

\newtheorem{theorem}{Theorem}[section]

\newtheorem{definition}[theorem]{Definition}

\newtheorem{lemma}[theorem]{Lemma}

\newtheorem{proposition}[theorem]{Proposition}

\newtheorem{property}[theorem]{Property}

\title{Approximate Green's Function for the Conductivity Equation With Conormal Coefficient
\thanks{This article is part of the author's 2016 Ph.D. thesis at the University of Rochester. The author thanks her advisor Prof. Allan Greenleaf, and the rest of the committee members: Prof. Alex Iosevich, Prof. Dan Geba and Prof. Marvin Doyley. Funds were provided by the National Science Foundation, NSF grants numbers DMS-0853892 and -1362271.}}

\author{%
Denitza Straub\footnote{%
UR Mathematics, 915 Hylan Building,
University of Rochester, RC Box 270138, Rochester, NY 14627, USA, e-mail: dgintcheva@gmail.com}
}
\date{}

\begin{document}
\maketitle

\begin{abstract}

We construct an approximate Green's function for $L_{\gamma}:=\div \gamma(x) \grad u(x)$, which belongs to a class of Fourier Integral Operators (FIOs) associated to two canonical relations. This leads to analysis of the composition of two FIOs, associated to a canonical relation with a zero section problem. The resulting composition is a sum of two FIOs, each associated to two intersecting canonical relations. 

\smallskip\noindent
{\em Keywords:} Microlocal analysis, divergence form operator, conductivity equation, Green's function, composition of FIOs with zero section. 

\smallskip\noindent
{\em MSC 35S30 \and MSC 58J40 \and MSC 31B20
}
\end{abstract}

\section{Introduction}

The first goal of this article is to explicitly construct a Green's function for the divergence form operator occurring in the conductivity equation:
\begin{equation}
L_{\gamma}:=\div \gamma(x) \grad u(x),
\end{equation}
where $\gamma(x)$ is assumed to be a real-valued conductivity, bounded below and above by a positive constant, and $u$ respresents a scalar-valued electrical potential. 

One instance in which finding an explicit form of the Green's function is necessary is for an application in medical imaging. In particular, consider a model class of hybrid inverse problems, in which the goal is to determine the unknown log-conductivity $\sigma$ in:
\begin{equation}\label{eq-cond-sigma}
L_{\sigma}u:=-\div e^{\sigma(x)} \grad u(x)=0, 
\end{equation}
in a domain $\Omega$ from boundary conditions:
\begin{equation}\label{eq-bdry}
u\vert_{\partial \Omega}=f, 
\end{equation}
and an internal functional, available for all of $\Omega$, which often takes the form: 
\begin{equation}
F(x; \sigma)=e^{\sigma}\abs{\grad u(x; \sigma)}^p, \quad 0<p\leq 2. 
\end{equation}
In particular, when $p=1$, the problem arises in Current Density Impedance Imaging (CDII) or Magnetic Resonance Electrical Impedance Tomography (MREIT). When $p=2$, the problem models Ultrasound Modulated Electrical Impedance Tomography (UMEIT). 

Kuchment and Steinhauer \cite{Kuchment2012} study the nonlinear map $\sigma \rightarrow F(\sigma)$, and its linearization at a smooth base point $\sigma_0$. They show that the Fr\'{e}chet derivative of the nonlinear map satisfies the equation:
\begin{equation}\label{eq-concluding1}
dF_{\sigma_0}(\rho)=\rho e^{\sigma_0}\abs{\nabla u_0}^p + p e^{\sigma_0} \frac{\nabla u_0 \cdot \nabla v(\rho)}{\abs{\nabla u_0}^{2-p}}, 
\end{equation}
where $v$ solves the boundary value problem:
\begin{equation} \label{eq-concluding2}
  \begin{split}
    L_{\sigma_0}v &= - \div (e^{\sigma_0} \nabla v) = \div (\rho e^{\sigma_0} \nabla u_0), \quad x \in \Omega, \\  
    v(x) & =0, \quad x \in \partial \Omega.
  \end{split}
\end{equation}
They show that this map is Fredholm, or semi-Fredholm, depending on $p$, when the assumed background log- conductivity $\sigma_0$ is in $C^{\infty}$.  The assumption that the background conductivity is $C^\infty$ is physically unrealistic, because the body or an industrial part contains many quite varied organs and tissues or components. One would not expect the conductivity to change smoothly across their boundaries. Furthermore, even within an organ, there are often microstructures not consistent with modeling by $C^\infty$ physical parameters. 

To achieve the greatest generalization, one would allow $\gamma=e^{\sigma}$ to be in $L^{\infty}$, bounded below by a positive constant. However, restricting $\gamma$ to a conormal class is still a good first step to generalizing the prior results. The Green's function becomes essential for solving the non-homogeneous equation (\ref{eq-concluding2}) for $v$, which is a key input for the linearized map (\ref{eq-concluding1}). 

The Green's function exists by reason of abstract functional analysis (\cite{Littman63}, Theorem 2.3). However, for the scale of hybrid problems in medical imaging, discussed above, we need a fairly concrete realization. This article provides an explicit construction when $\gamma$ is restricted to be a conormal distribution. In particular, we assume $\gamma$ can be written as a conormal distribution of order $\mu<-1$ for $S$, where $S$ is a codimension one submanifold of $\mathbb{R}^n$ with defining function $h(x)=0$.
This means, we can write:
\begin{equation}\label{eq-gamma-intro} 
 \gamma(y)=\int e^{i h(y) \cdot \theta} a_{\mu}(y, \theta) d\theta, \quad a_{\mu}\in S^{\mu}(\mathbb{R}^n \times \mathbb{R}^1\setminus 0),
\end{equation}
where $a_{\mu}$ is a standard \emph{H\"ormander symbol} of type $\rho=1, \delta=0$ \cite{Hormander71}. A conormal distribution for $S$ implies smoothness away from $S$, while the order on the symbol $\mu<-1$ means $\gamma$ is continuous across the interface $S$, but is in $C^{0,\epsilon}(\Omega)$ for $\epsilon=-\mu-1$.

For the explicit construction of the Green's function, we break up $L_{\gamma}$ into a sum of five operators,
\begin{equation} \label{eq-lg-as-sum}
L_{\gamma}=A_1^2+A_1^1+A_2+A_3^{\mu}+A_3^{\mu+1}
\end{equation}
In the above expression $A_2 \in I^{\mu+\frac{5}{2}} (C_{\Sigma})$, a standard Fourier integral operator (FIO), associated with the canonical relation, represented by the flowout of $T^*\RR^n\vert_{S}$ by the Hamiltonian vector field $H_h=-\nabla h \cdot \frac{\partial}{\partial \xi}$, namely,
\begin{equation}\label{eq-C-sigma}
C_{\Sigma}=\{(x, \xi, y, \eta) \vert x=y, \eta=\xi-\nabla h(x) t, x\in S, t\in \mathbb{R}^1, (\xi, t)\in \mathbb{R}^{n+1} \setminus 0\}. 
\end{equation} 
The rest of the operators are Fourier integral operators associated to two cleanly intersecting canonical relations, \cite{Melrose79}, \cite{Guillemin81}, \cite{Mendoza82}.
More detail on these classes is provided in section \ref{ipl}. The construction of an approximate Green's function to $L_{\gamma}$ is done in two parts. First, we construct an approximate inverse $B$ to $A_1^2+A_1^1$ in stages and iterations. This part is captured in Theorem \ref{thm-chi1inversion}. Next, we formulate a composition theorem, Theorem \ref{thm-composition}, which enables us to classify the compositions of $A_2$, $A_3^{\mu}$, and $A_3^{\mu+1}$ with $B$, respectively. It turns out that they all belong to the same residual class as $(A_1^2+A_1^1)\circ B - Id$, namely $I^{2\mu+\frac{3}{2}}(C_{\Sigma})$, which is smoothing on the scale of Sobolev spaces, as long as $\mu<-1$. Therefore, the constructed $B$ is indeed an approximate Green's function to $L_{\gamma}$. 

In section \ref{Ch-FIO} we analyse the composition of two FIOs, associated to $C_{\Sigma}$. This canonical relation $C_{\Sigma}$ can intersect the zero section both on the left (when $\xi=0$) and on the right (when $\eta=0$), but not both at the same time. The standard composition theorems assume canonical relations avoiding the zero section altogether, with an exception being \cite{Greenleaf01}. We formulate and prove Theorem \ref{thm-FIO-0sec}. The result of the composition is a sum of two FIOs, each associated to two canonical relations.

The paper is organized as follows. In section \ref{ipl} we review the $I^{p,l}$ classes, in section \ref{Ch-Green} we construct the Green's function to $L_{\gamma}$, and in section \ref{Ch-FIO} we analyze the composition of FIOs associated to $C_{\Sigma}$.

\section{$I^{p,l}$ Classes} \label{ipl}

Let $S^m$ denote the standard H\"ormander symbols of type $(1,0)$ of order $m \in \mathbb{R}$. 
A \emph{pure product type symbol of order $m$, $m'$} (\cite{Guillemin81}) is a $ C^{\infty}$ function $p(x, \xi, \theta)$ on $(\RR^n \times \RR^N \times \RR^M)$, such that for all compact sets $K$, and all multi-indices $\alpha, \beta, \gamma$, 
\begin{equation}
 \abs{\partial_x^{\gamma}\partial_{\xi}^{\alpha}\partial_{\theta}^{\beta}p(x,\xi, \theta)}\leq C_{\alpha, \beta, \gamma, K} \langle \xi \rangle^{m-\abs{\alpha}}\langle \theta \rangle^{m'-\abs{\beta}}, \quad x \in K, \xi \in \RR^N, \theta \in \RR^{M},  
\end{equation}
where $\langle \xi \rangle:=(1+\abs{\xi}^2)^{\frac{1}{2}}$ is known as the Japanese bracket of $\xi$.

A \emph{symbol-valued symbol of order $m$, $m'$} (\cite{Greenleaf90}) is a $ C^{\infty}$ function $p(x, \xi, \theta)$ on $(\RR^n \times \RR^N \times \RR^M)$, denoted  $p \in S^{m,m'}(\RR^n, (\RR^N \setminus 0), \RR^M)$, such that for all compact sets $K$, and all multi-indices $\alpha, \beta, \gamma$, 
\begin{equation}\label{eq-smm'}
 \abs{\partial_x^{\gamma}\partial_{\xi}^{\alpha}\partial_{\theta}^{\beta}p(x,\xi, \theta)}\leq C_{\alpha, \beta, \gamma, K} \langle \xi, \theta \rangle^{m-\abs{\alpha}}\langle \theta \rangle^{m'-\abs{\beta}}, \quad x \in K, \xi \in \RR^N, \theta \in \RR^{M}.  
\end{equation}

 We define a Fourier integral distribution associated with two cleanly intersecting Lagrangians, as per Mendoza \cite{Mendoza82}. Although this is not the original definition given by Melrose and Uhlmann \cite{Melrose79}, and Guillemin and Uhlmann \cite{Guillemin81}, it is equivalent. First we start with the definition of a multiphase function.  
 
 \begin{definition}\label{def-fid}
 Let $\Lambda_0$ and $\Lambda_1$ be two Lagrangians in $T^*\RR^n\setminus 0$, intersecting cleanly with excess $e=k$. Then for any $\lambda \in \Lambda_0 \cap \Lambda_1$, there exists a real-valued nondegenerate multiphase function $\phi(x,\theta, \sigma)$ on $\RR_x^n \times (\RR_{\theta}^{N-k} \setminus 0) \times \RR_{\sigma}^k$, homogeneous of degree 1 in $(\theta, \sigma)$, which microlocally parametrizes the pair of Lagrangians. In particular, the full function $\phi(x,\theta, \sigma)$ parameterizes the Lagrangian $\Lambda_1$, and $\phi_0(x, \theta) = \phi(x, \theta, 0)$ parameterizes the Lagrangian $\Lambda_0$.  
  \end{definition}
  
  \begin{definition}
  The Fourier integral distributions $I^{p,l}(X; \Lambda_0, \Lambda_1)$, associated with the two Lagrangians $\Lambda_0$ and $\Lambda_1$, are defined as all locally finite sums of expressions of the form:
  \begin{equation}
  \begin{split}
   & u(x) = \int_{\RR^N}e^{i\phi(x, \theta, \sigma)}a(x, \theta, \sigma) d\theta d\sigma,
   \quad a \in S^{M, M'}(\RR^n_x, \RR^{N-k}_{\theta}\setminus 0, \RR^k_{\sigma}),\\
   & \text{ where } 
   M=(p+l)-\frac{N-k}{2}+\frac{n}{4} \text{ and } M'=-l-\frac{k}{2},
   \end{split}
  \end{equation}
  over all possible multiphase functions parameterizing the pair of Lagrangians.
\end{definition}

There is an iterated regularity characterization, due to Melrose (see \cite{Greenleaf90}), of Fourier integral distributions associated to two Lagrangians, although the orders $p$ and $l$ are not evident from this characterization. 

\begin{theorem}\label{thm-itreg-2}
 A Fourier integral distribution $\displaystyle u \in \bigcup_{p,l\in\RR} I^{p,l}(X; \Lambda_0, \Lambda_1)$ if and only if there exists $s_0 \in \RR$ with $u\in H^{s_0}_{loc}(X)$, such that for all $M \in \mathbb{N}$, and all classical first order pseudodifferential operators $P_1, P_2, ... P_M \in \Psi^1_{cl}$, \emph{characteristic} for $\Lambda_0 \cup \Lambda_1$, i.e., $\sigma_{prin} (P_j)\equiv 0$ on $\Lambda_0 \cup \Lambda_1$, 
 \begin{equation} \label{eq-itreg-2}
 P_1P_2...P_Mu \in H^{s_0}_{loc}(X).
 \end{equation}
\end{theorem}

Note the following property from Guillemin and Uhlmann{\cite{Guillemin81}}, Proposition 6.2, used later in Chapter \ref{Ch-Green}.
\begin{property}\label{prop-ipl-caps}
With $I^{p,l}=I^{p,l}(X; \Lambda_0, \Lambda_1)$ we have\\
 (i)$\quad \bigcap_{l} I^{p,l}=I^p(X, \Lambda_1)$\\
 (ii)$\quad \bigcap_{p} I^{p,l}=C_0^{\infty}(X)$.
\end{property}

The authors also derive the following inclusions.

\begin{property}\label{prop-inclusions}
Let $u \in I^{p,l}(X; \Lambda_0, \Lambda_1)$. Then:\\
(i)$\quad u\vert_{\Lambda_0\setminus \Lambda_1} \in I^{p+l}(X; \Lambda_0 \setminus \Lambda_1)$, and \\
(ii)$\quad u\vert_{\Lambda_1\setminus \Lambda_0} \in I^{p}(X; \Lambda_1 \setminus \Lambda_0)$.
\end{property}

The last property means that on each of the two Lagrangians, away from the intersection, these classes behave just like Fourier integral distributions, associated with a single Lagrangian. Therefore, each one has a well-defined principal symbol in the sense of H\"ormander. We will introduce a principal symbol in the sense of \cite{Guillemin81}, but first we need a few of their definitions. Their Definitions 5.1, 5.2, paraphrased with reindexing below, state what it means for a class of functions defined on $X\setminus Y$ to ``blow up like homogeneous functions of degree $r$'' at $Y$. First for $X=R^n$ and $Y=0$ these are defined as follows:
\begin{definition}
 Let $f$ be a smooth function on $\RR^n\setminus 0$. Then $f\in \Theta^l \Leftrightarrow \abs{D^{\alpha}f} \leq C_{\alpha}\abs{x}^{l-\abs{\alpha}}$ for all multi-indices, $\alpha$. We will say that $f$ \emph{has a singularity of order $r$ at $0$ if,} for every $N\in \mathbb{N}$, there exists a sequence of homogeneous functions, $f_{i}$, $0 \leq i \leq N$, on $\RR^n \setminus 0$, $f_i$ being homogeneous of degree $i+r$, such that $f-\sum_{i=0}^N f_i \in \Theta^{r+N+1} \cup C^{\infty}(\RR^n)$. 
\end{definition}
Now we describe the more general case, where $X=\RR^n$, and $Y$ is a submanifold of $\RR^n$ defined by the equations $x_{k+1}=...=x_n=0$. Let $y_1,...,y_k$ denote the first $k$ coordinates of $\RR^n$, and $t_1,...,t_{n-k}$ the remaining $n-k$ coordinates. 
\begin{definition}
 Let $f$ be a smooth function on $\RR^n\setminus Y$. We will say that $f$ \emph{has a singularity of order $r$ along $Y$ if,} for every $N\in \mathbb{N}$, there exists a sequence of smooth functions, $f_i(t, y)$, $0\leq i\leq N$, on $\RR^n \setminus Y$, each $f_i$ being homogeneous of degree $r+i$ in $t$, and a smooth function $g$, such that $f-g-\sum_{i=0}^N f_i \in \Theta^{r+N+1}$ as a function of $t$, uniformly in $y$. The space of all such $f$'s will be denoted by $R^r(X,Y)$.
\end{definition}
For a distribution $u \in I^{p,l}(X; \Lambda_0, \Lambda_1)$, the H\"ormander principal symbol on $\Lambda_0\setminus \Lambda_1$ is denoted by $\sigma_0$ and on $\Lambda_1\setminus \Lambda_0$ is denoted by $\sigma_1$. Let $k$ be the codimension of the intersection $\Sigma=\Lambda_0 \cap \Lambda_1$ in either Lagrangian. The following proposition, which we use in Sections \ref{Ch-Green} and \ref{Ch-FIO}, is a rephrased Proposition 5.4 \cite{Guillemin81}, in order to avoid irrelevant technicalities concerning half bundles and Maslov densities. It defines the order of blowup, as the intersection is approached:
\begin{proposition}
$\sigma_0 \in R^{l'}(\Lambda_0, \Sigma)$ and $\sigma_1 \in R^{-l'-k}(\Lambda_1, \Sigma)$, where $l'=l-k/2$.   
\end{proposition}

Guillemin and Uhlmann define a map $\alpha$ from the principal symbol to the leading term $f_0(t,y)$, which is their notion of principal symbol $\sigma_{prin}$. They deduce that the leading terms of the two symbols are related by the Fourier transform, namely $\alpha(\sigma_1)=F.T.\, \alpha(\sigma_0)$.

We continue with the definition of Fourier integral operators, associated with two canonical relations.

\begin{definition} \label{def-FIO-ipl}
 Let $X$ be a smooth manifold of dimension $n_x$, $Y$ be a smooth manifold of dimension $n_y$. Let $C_0$ be a canonical relation, with $C_0=\Lambda_0' \subset T^* \RR^{n_x+n_y}\setminus 0$, where $\Lambda_0$ is a smooth conic Lagrangian, and let $C_1$ be a canonical relation, with $C_1=\Lambda_1' \subset T^* \RR^{n_x+n_y}\setminus 0$, where $\Lambda_1$ is a smooth conic Lagrangian, and assume $\Lambda_0$ and $\Lambda_1$ intersect cleanly in codimension $k$.  
We say the operator $A: \mathcal D(Y)$ to $\mathcal D'(X)$ is in $I^{p,l}(X, Y; C_0, C_1)$ if its kernel
$K_A \in I^{p,l}(X\times Y; C_0', C_1')$.
\end{definition}

Now that there is a notion of the principal symbol on either Lagrangian away from the intersection, we state a property, which we use over and over again in the construction of the Green's function in Chapter \ref{Ch-Green}.
\begin{property}\label{prop-ipl-residual}
 Let $A \in I^{p,l}(X, Y; C_0, C_1)$ and $\sigma_{prin}(A)=0$.
 
 Then $A \in I^{p-1,l}(X, Y; C_0, C_1)+I^{p,l-1}(X, Y; C_0, C_1)$.
\end{property}

\section{Green's Function Construction}\label{Ch-Green}
In this section, we construct a Green's function for $L_{\gamma}=\div (\gamma(x) \nabla)$. The Green's function exists by reason of abstract functional analysis (\cite{Littman63}, Theorem 2.3), but to be useful for the inverse problem, we need a fairly concrete realization. We assume $\gamma$ can be written as a conormal distribution of order $\mu<-1$ for $S$, where $S$ is a codimension 1 submanifold of $\mathbb{R}^n$ with defining function $h(x)=0$. This means, we can write:
\begin{equation} \label{eq-gamma}
 \gamma(y)=\int e^{i h(y) \cdot \theta} a_{\mu}(y, \theta) d\theta, \quad a_{\mu}\in S^{\mu}(\mathbb{R}^n \times \mathbb{R}^1\setminus 0)
\end{equation}

Define the operator $\dot{\nabla}_{\gamma} \phi := \div (\gamma(x)\phi(x))$ mapping vector-valued functions to scalar-valued ones. Then:
\begin{equation} \label{eq-nabladot}
 \dot{\nabla}_{\gamma} \phi(x) = \int e^{i(x-y)\cdot \xi} \gamma(y) (i\xi)\cdot \phi(y) \, dy \, d\xi 
\end{equation}
and
\begin{equation} \label{eq-lgammaf}
  \begin{array}{ll}
 L_{\gamma} f(x) = \dot{\nabla}_{\gamma} (\nabla f (x)) &= \int e^{i(x-y)\cdot \xi} \gamma(y) (i\xi)\cdot \nabla f (y) \, dy \, d\xi\\
							  &= \int e^{i[(x-y)\cdot \xi+(y-z)\cdot \eta]} \gamma(y) (i\xi)\cdot (i\eta) f (z) \,dz \, dy \, d\xi \,d\eta.
  \end{array}
 \end{equation}
 Substituting (\ref{eq-gamma}) into (\ref{eq-lgammaf}) leads to
  \begin{equation} \label{eq-lgammaflong}
    L_{\gamma} f(x) = -\int e^{i[(x-y)\cdot \xi+(y-z) \cdot \eta + h(y)\theta]} (\xi \cdot \eta) a_{\mu}(y, \theta) f(z) \, dz \, dy \, d\xi \, d\eta \, d\theta. 
  \end{equation}
The kernel of $L_{\gamma}$ can be represented as:
 \begin{equation} \label{eq-klgammalong}
   K_{L_{\gamma}} (x,z) = \int e^{i[(x-y)\cdot \xi+(y-z) \cdot \eta + h(y)\theta]} (\xi \cdot \eta) a_{\mu}(y, \theta) \, dy \, d\xi \, d\eta \, d\theta.
 \end{equation}

 Now, we reduce the number of variables in the integration via stationary phase in $\eta$ and $y$. The phase function is:
 \begin{equation} \label{eq-phi1}
  \Phi(x,z,\xi, \theta; \eta, y) = (x-y)\cdot \xi+(y-z) \cdot \eta + h(y)\theta
 \end{equation}
 
 Since, 
 \begin{equation*}
  \begin{array}{l}
   \Phi_{\eta}=y-z\\
   \Phi_{y}=-\xi+\eta+\nabla h(y) \theta,
  \end{array}
 \end{equation*}
the unique critical point is $y=z$ and $\eta = \xi - \nabla h(y) \theta$.
The Hessian and inverse Hessian of $\Phi$ are respectively, 
$$\cal{H}=\begin{bmatrix} 0 & I_n \\ I_n & -\theta H(y) \end{bmatrix} \quad \text{and} \quad  \cal{H}^{\text{-1}}=\begin{bmatrix} -\theta H(y) & I_n \\ I_n & 0 \end{bmatrix},$$
where $H$ is the Hessian of the defining function $h$. Since $H$ is nonsingular, the critical point is nondegenerate.

We can substitute the critical point into (\ref{eq-klgammalong}), provided there is a valid asymptotic expansion, which we justify below, and obtain: 
\begin{equation} \label{eq-klgamma}
  K_{L_{\gamma}} (x,z) \sim \int e^{i[(x-z)\cdot \xi+ h(z)\theta]} (\xi \cdot (\xi - \nabla h(z)\theta)) a_{\mu}(z; \theta)  \, d\xi \, d\theta.
\end{equation}
The amplitude of the integral is (\ref{eq-klgammalong}) is
\begin{equation}\label{eq-amp-a}
 a(\xi, \eta, y, \theta)=(\xi \cdot \eta) a_{\mu}(y; \theta)
\end{equation}
and its behavior at the critical point can be estimated, as a sum of two pure-product type symbols with orders, given by the exponents in
\begin{equation} \label{eq-a-crit}
\begin{array}{ll}
 \abs{a}_{\vert{\text{crit.pt.}}} &= \abs{\xi \cdot (\xi - \nabla h(z)\theta) a_{\mu}(z; \theta)} \\ 
 &\leq \abs{\xi \cdot \xi a_{\mu}(z; \theta)} + \abs{\xi \cdot \nabla h(z)\theta a_{\mu}(z; \theta)} \\
 &\lesssim \abs{\xi}^2 \langle \theta \rangle^{\mu}+ \abs{\xi}\langle \theta \rangle^{\mu+1} \\
 &\lesssim \langle\xi\rangle^2 \langle \theta \rangle^{\mu} + \langle\xi\rangle\langle \theta \rangle^{\mu+1}
 \end{array}
\end{equation}
In order to justify the asymptotic expansion we need to analyze the action of the operator 
\begin{equation} \label{eq-lop}
L = \nabla_{\eta, y} \cdot \mathcal{H}^{-1} \nabla_{\eta, y} = -\theta \sum_{i,j=1}^n H_{ij}(y)\frac{\partial^2}{\partial \eta_i\partial \eta_j}+2\sum_{i=1}^n \frac{\partial^2}{\partial \eta_i\partial y_i}
\end{equation}
on the amplitude $a$ and evaluate at the critical point to ensure decay. We put an upper bound on the size of $\abs{La}$ at the critical point:
\begin{equation} \label{eq-a-deriv-crit}
 \abs{La}_{\vert\text{crit.pt.}}\lesssim \sum_{i=1}^n \abs{\xi_i} \langle \theta \rangle^{\mu} \lesssim \langle \xi \rangle \langle \theta \rangle^{\mu}.
\end{equation}
Expression (\ref{eq-a-deriv-crit}) shows decay relative to either term in (\ref{eq-a-crit}), which justifies the stationary phase. Moreover, (\ref{eq-a-crit}) shows that the amplitude is a sum of two pure product type symbols. Using this fact, and switching notation to write $\sigma$, instead of $\xi$, and $\tau$, instead of $\theta$, we rewrite (\ref{eq-klgamma}) as 
\begin{equation} \label{eq-klgamma-sumpureproduct}
 K_{L_{\gamma}} (x,z) \sim \int e^{i[(x-z)\cdot \sigma+ h(z)\tau]} (a_2(\sigma)a_{\mu}(z; \tau)+a_1(\sigma)a_{\mu+1}(z; \tau))\, d\sigma \, d\tau.
\end{equation}

To continue with our analysis, we will introduce a partition of unity, which will allow us to look at the amplitude as a sum of four symbol-valued symbols and a standard symbol. 

Introduce the following open cover of $S^1$ in $\abs{\sigma}$, $\abs{\tau}$ space:
\begin{equation} \label{eq-opencover-u}
\begin{split}
\Big\{\{\frac{\abs{\tau}}{\abs{\sigma, \tau}}<\delta_2\}, \{\delta_1<\frac{\abs{\tau}}{\abs{\sigma, \tau}}<\delta_4\}, \{\delta_3<\frac{\abs{\tau}}{\abs{\sigma, \tau}}\} \Big\}, \\
  \text{for }0<\delta_1<\delta_2<\frac{1}{\sqrt{5}}<\frac{2}{\sqrt{5}}<\delta_3<\delta_4<1.
\end{split}
 \end{equation}

There exists a smooth partition of unity, subordinate to this cover with 
\begin{equation}\label{eq-partition-chi}
 \chi_1+\chi_2+\chi_3\equiv 1 \quad \text{on } S^1.
\end{equation}
This partition of unity can be extended to all of $\mathbb{R}^2 \setminus 0$, demanding it is homogeneous of degree zero, and then to all of $\mathbb{R}^{n+1}\setminus 0$, since $\sigma \in \mathbb{R}^n$ and $\tau \in \mathbb{R}^1$.
\begin{equation} \label{eq-klgamma-chi}
\begin{split}
&K_{L_{\gamma}} (x,z) \\
&\sim \int e^{i[(x-z)\cdot \sigma+ h(z)\tau]}(\chi_1+\chi_2+\chi_3)(a_2(\sigma)a_{\mu}(z; \tau)+a_1(\sigma)a_{\mu+1}(z; \tau))\, d\sigma \, d\tau \\ 
&=\int e^{i\phi} \chi_1 a + \int e^{i\phi} \chi_2 a + \int e^{i\phi} \chi_3 a
\end{split}
\end{equation}

In order to compare directly the size of $\abs{\sigma}$ and $\abs{\tau}$ on the support of each $\chi_i$, we introduce the constants $c_i=\frac{\delta_i}{\sqrt{1-\delta_i^2}}$, for 
$i \in {1,2,3,4}$. Then it is easily seen that $c_1<c_2<0.5<2<c_3<c_4<\infty$.

On the support of $\chi_1$, $\abs{\tau}<c_2\abs{\sigma}$, which implies that the amplitude can be treated as the sum of two symbol-valued symbols, which we can write as:
\begin{equation} \label{eq-symb-chi1}
 a_{2,\mu}(z; \sigma, \tau) \in S(\mathbb{R}_z^n; \mathbb{R}_{\sigma}^n \setminus 0, \mathbb{R}_{\tau}^1)\quad \text{and} \quad 
 a_{1,\mu+1}(z; \sigma, \tau) \in S(\mathbb{R}_z^n; \mathbb{R}_{\sigma}^n \setminus 0, \mathbb{R}_{\tau}^1).
 \end{equation}
The integrals associated with these symbols will be written as:
\begin{equation} \label{eq-A12}
K_{A_1^2}(x,z)=\int e^{i[(x-z)\cdot \sigma + h(z) \cdot \tau]}a_{2,\mu}(z; \sigma, \tau)\, d\sigma \,d\tau \in I^{\mu + \frac{5}{2}, -\mu-\frac{1}{2}}(\Delta', C_{\Sigma}'), \abs{\tau}<c_2\abs{\sigma},
 \end{equation}
\begin{equation} \label{eq-A11}
 K_{A_1^1}(x,z)=\int e^{i[(x-z)\cdot \sigma + h(z) \cdot \tau]}a_{1,\mu+1}(z; \sigma, \tau)\, d\sigma \,d\tau \in I^{\mu + \frac{5}{2}, -\mu-\frac{3}{2}}(\Delta', C_{\Sigma}'), \abs{\tau}<c_2\abs{\sigma},
\end{equation}
where we use the standard notation for $I^{p,l}$ classes, associated with two cleanly intersecting Lagrangians. In particular, the two Lagrangians correspond to the canonical relations $\Delta$ and $C_{\Sigma}$, where $\Delta$ is the diagonal and $C_{\Sigma}$ is the flowout of $T^*\RR^n\vert_{S}$ by the Hamiltonian vector field $H_h=-\nabla h \cdot \frac{\partial}{\partial \xi}$, namely,
\begin{equation*}
 \Delta=\{(x, \xi, y, \eta) \vert x=y, \xi=\eta\in \mathbb{R}^n \setminus 0\}, 
\end{equation*}
\begin{equation*}
C_{\Sigma}=\{(x, \xi, y, \eta) \vert x=y, \eta=\xi-\nabla h(x) t, x\in S, t\in \mathbb{R}^1, (\xi, t)\in \mathbb{R}^{n+1} \setminus 0\}. 
\end{equation*}

On the support of $\chi_2$, $c_1\abs{\sigma}<\abs{\tau}<c_4\abs{\sigma}$, which means that $\abs{\sigma}$ and $\abs{\tau}$ are comparable, i.e., $\abs{\sigma}\sim \abs{\tau}$ and the amplitude $a$ can be treated as a standard symbol in all $\mathbb{R}^{n+1}$ variables, namely,
\begin{equation}\label{eq-symb-chi2}
 a \in S(\mathbb{R}_z^n; R_{\sigma, \tau}^{n+1}\setminus 0).
\end{equation}
The integral will be denoted as:
\begin{equation} \label{eq-A2}
 K_{A_2}(x,z)=\int e^{i[(x-z)\cdot \sigma + h(z) \cdot \tau]}a_{\mu+2}(z; \langle \sigma, \tau \rangle)\, d\sigma \,d\tau \in I^{\mu + \frac{5}{2}}(C_{\Sigma}'), c_1\abs{\sigma}<\abs{\tau}<c_4\abs{\sigma}.
\end{equation}

On the support of $\chi_3$, $c_3\abs{\sigma}<\abs{\tau}$, which implies that the amplitude can be treated as the sum of two symbol-valued symbols, which can be written as:
\begin{equation} \label{eq-symb-chi3}
 a_{\mu, 2}(z; \sigma, \tau) \in S(\mathbb{R}_z^n; \mathbb{R}_{\tau}^1\setminus 0, \mathbb{R}_{\sigma}^n)\quad \text{and} \quad 
 a_{\mu+1, 1}(z; \sigma, \tau) \in S(\mathbb{R}_z^n; \mathbb{R}_{\tau}^1\setminus 0, \mathbb{R}_{\sigma}^n).
 \end{equation}
 The integrals associated to these symbols will be denoted as:
\begin{equation} \label{eq-A3mu}
 K_{A_3^{\mu}}(x,z)=\int e^{i[(x-z)\cdot \sigma + h(z) \cdot \tau]}a_{\mu, 2}(z; \tau, \sigma)\, d\sigma \,d\tau \in I^{\mu + \frac{5}{2}, -2-\frac{n}{2}}(C_0', C_{\Sigma}'), c_3\abs{\sigma}<\abs{\tau},
\end{equation}
\begin{equation} \label{eq-A3mu1}
 K_{A_3^{\mu+1}}(x,z)=\int e^{i[(x-z)\cdot \sigma + h(z) \cdot \tau]}a_{\mu+1,1}(z; \tau, \sigma)\, d\sigma \,d\tau \in I^{\mu + \frac{5}{2}, -1-\frac{n}{2}}(C_0', C_{\Sigma}'), c_3\abs{\sigma}<\abs{\tau}.
\end{equation}
Here,
\begin{equation*}
C_0=\{(x, 0, y, \nabla h(x) t) \vert x\in \mathbb{R}^n, y\in S, t\in \mathbb{R}^1\setminus 0\}. 
\end{equation*}
Now, we are ready to begin the construction of a Green's function to $K_{L_{\gamma}}$. Our first focus will be on approximately inverting $A_1^2+A_1^1$ in stages.

\bigskip
First, we will make the following remarks. We say that a canonical relation $C={(x,\xi, y, \eta)}$ \emph{has a $0$-section problem on the left} if $\xi$ can take the value $0$, and a $0$-section problem \emph{on the right} if $\eta$ can take the value $0$. By definition of a canonical relation, both cannot be zero at the same time. For $C_{\Sigma}$ in general, we only know that $(\xi,t)\neq ({\bf 0},0)$, but that does not prevent either $0$-section from occurring. A $0$-section problem on the left occurs if $\xi=0, t\neq 0$ and a $0$-section problem on the right occurs if $\xi \neq 0$, but $\xi$ is collinear with $\nabla h(x)$, $\xi=t\nabla h(x)$.

However, on the support of $\chi_1$, there is no $0$-section problem. This is because $\abs{\tau}<c_2\abs{\sigma}$ implies $\abs{t}<c_2\abs{\xi}$. This means $\abs{\xi} \neq 0$, otherwise $\abs{t}=0$ as well, and $(\xi, t)=0$, a violation. This means no $0$-section problem on the left. On the other hand, assuming $\abs{\nabla h(x)}\equiv 1$ for all $x$ on $S$, 
$\abs{\eta}=\abs{\xi - \nabla h(x) t} \geq \abs{\xi}-\abs{t} > (1-c_2)\abs{\xi}>0$, since $c_2<\frac{1}{2}$. Therefore, there is no $0$-section problem on the right. 

Also, notice that composition of two (or an arbitrary number of) operators with canonical relation $C_{\Sigma}$ on $supp(\chi_1)$ leads to no $0$-section problem, because $(x, \xi, y, \eta) \in C_{\Sigma} \circ C_{\Sigma}$ on $\chi_1$ if and only if there exists  $(z,\zeta)$, such that $(x, \xi, z, \zeta) \in C_{\Sigma}$ and $(z, \zeta, y, \eta) \in C_{\Sigma}$, but this means $x=z=y$ and $\abs{\eta}>(1-c_2)\abs{\zeta}>(1-c_2)^2\abs{\xi}>0$. Now that the zero-section has been avoided, the  composition of operators in $I^{p,l}(\Delta, C_{\Sigma})$ classes can be accomplished with the use of a result in \cite{Antoniano85}, namely:
\begin{theorem}\label{thm-Antoniano}
$I^{p,l} (\Delta, C_{\Sigma}) \circ I^{p',l'} (\Delta, C_{\Sigma}) \subset I^{p+p'+\frac{1}{2},l+l'-\frac{1}{2}} (\Delta, C_{\Sigma}).$
\end{theorem}

In order to invert $(A_1^2+A_1^1)$, we will build a sequence of operators $B_i^j \in I^{\mu-i-\frac{1}{2},-\mu-j+\frac{1}{2}}(\Delta, C_{\Sigma})$ with the following kernel representations:
\begin{equation}
 K_{B_i^j}(z,y)=\int e^{i[(z-y)\cdot \tilde\sigma-h(z)\tilde\tau]}b_{-i-j, \mu+j-1}(z; \tilde\sigma, \tilde\tau) \, d\tilde\sigma \, d\tilde\tau, \quad \abs{\tilde \tau}<c_2\abs{\tilde \sigma}.
\end{equation}

We will prove the following theorem, which will lead to the construction of the right inverse $B$ to $(A_1^2+A_1^1)$, and consequently, to the operator $L_{\gamma}$.
\begin{theorem} \label{thm-chi1inversion}
For the operators 
$$A_1^2 \in I^{\mu+\frac{5}{2}, -\mu-\frac{1}{2}}(\Delta, C_{\Sigma}),$$ and 
$$A_1^1 \in I^{\mu+\frac{5}{2}, -\mu-\frac{3}{2}}(\Delta, C_{\Sigma}),$$ supported close to the diagonal, there exists an operator $B \in I^{\mu-\frac{3}{2}, -\mu-\frac{1}{2}}(\Delta, C_{\Sigma})$, and operators $B_i \in  I^{\mu-i-\frac{1}{2}, -\mu-\frac{1}{2}}(\Delta, C_{\Sigma}), i\in \mathbb{N}$, such that
\begin{equation}
  B \sim \sum_{i=1}^{\infty} B_i, 
\end{equation}
and each
\begin{equation}
  B_i \sim \sum_{j=1}^{\infty} B_i^j, j \in \mathbb{N},
\end{equation}
where each $B_i^j \in I^{\mu-i-\frac{1}{2}, -\mu-j+\frac{1}{2}}(\Delta, C_{\Sigma})$,
such that the following identity always holds:
\begin{equation}
(A_1^2+A_1^1) \circ (\sum_{i=1}^M B_i+\sum_{j=1}^N B_{M+1}^j)-Id=\sum_{j=1}^N E_{M+1}^{j,1}+(E_{M+1}^{N,2}+E_{M+1}^{N,3}),
\end{equation}
modulo an operator, which is smoothing on the scale of Sobolev spaces, where 
\begin{equation}
E_i^{j, 1} \in  I^{2\mu-i+\frac{3}{2}, -2\mu-j-\frac{1}{2}}(\Delta, C_{\Sigma})
\end{equation}
and 
\begin{equation}
E_i^{j, 2}, E_i^{j, 3} \in I^{2\mu-i+\frac{5}{2}, -2\mu-j-\frac{3}{2}}(\Delta, C_{\Sigma})
\end{equation}
are error terms. 
The operator $B$ is a right pseudoinverse of $(A_1^2+A_1^1)$, in the sense that
\begin{equation}
(A_1^2+A_1^1) \circ B -Id \in I^{2\mu+\frac{3}{2}}(C_{\Sigma}). 
\end{equation}
\end{theorem}

 Using the composition calculus for flowouts, and their estimates, covered by the clean intersection calculus of Duistermaat-Guillemin \cite{Duistermaat75} and Weinstein \cite{Weinstein74}, we know that $I^{2\mu+\frac{3}{2}}(C_{\Sigma}):H^s \rightarrow H^{s-(2\mu+\frac{3}{2})-\frac{1}{2}}$, or $H^s \rightarrow H^{s-(2\mu+2)}$, which is a smoothing operator as long as $2\mu+2<0$, or $\mu<-1$, which has been assumed all along. 

Before we turn to the proof of Theorem \ref{thm-chi1inversion}, detailed in Section \ref{sec-Inversion}, we would like to understand the compositions $A_2 \circ B$ and $A_3^{\mu}\circ B$, $A_3^{\mu+1}\circ B$. It turns out that these compositions lie in the same residual class, as was obtained from the theorem, namely $I^{2\mu+\frac{3}{2}}(C_{\Sigma})$. Thus, our constructed $B \in I^{\mu-\frac{3}{2}, -\mu-\frac{1}{2}}$ with symbol $b_{-2,\mu}$, such that $(A_1^2+A_1^1)\circ B=Id \quad \text{mod }I^{2\mu+\frac{3}{2}}(C_{\Sigma})$ can be viewed as the Green's function to the whole operator $L_{\gamma}=A_1^2+A_1^1+A_2+A_3^{\mu}+A_3^{\mu+1}$.

The compositions $A_2 \circ B$ and $A_3^{\mu}\circ B$, $A_3^{\mu+1}\circ B$ are captured by the following theorem: 
\begin{theorem}\label{thm-composition}
Let $A \in I^{p,l}(C_0, C_{\Sigma})$ or $A \in I^p(C_{\Sigma})$, and let $B \in I^{p',l'}(\Delta, C_{\Sigma})$ with both operators having restricted support. In particular, for constants
$c_1<c_2<\frac{1}{2}<2<c_3<c_4<\infty$, assume the kernel of $A$ has the representation:
\begin{equation*}
\begin{split}
K_A(x, z)=\int e^{i[(x-z)\cdot \sigma + h(z) \tau]} a_{M,M'}(z; \tau, \sigma) \, d\sigma \, d\tau, \\
\text{ where } a_{M,M'}\in S^{M,M'}(\mathbb{R}_z^n; \mathbb{R}_{\tau}^1\setminus 0, \mathbb{R}_{\sigma}^n) \text{ and }  c_3\abs{\sigma}<\abs{\tau},
\end{split}
\end{equation*}
or
\begin{equation*}
\begin{split}
K_A(x, z)=\int e^{i[(x-z)\cdot \sigma + h(z) \tau]} a_{M}(z; \tau, \sigma) \, d\sigma \, d\tau, \\
\text{ where } a_{M}\in S^{M}(\mathbb{R}_z^n; \mathbb{R}_{\tau, \sigma}^{n+1}\setminus 0) \text{ and } c_1\abs{\sigma}<\abs{\tau}<c_4\abs{\sigma},
\end{split}
\end{equation*}
and the kernel of $B$ has the representation:
\begin{equation*}
\begin{split}
K_B(z, y)=\int e^{i[(z-y)\cdot \tilde\sigma - h(z) \tilde\tau]} b_{\tilde M,\tilde M'}(z; \tilde\sigma, \tilde\tau) \, d\tilde\sigma \, d\tilde\tau, \\
\text{ where } b_{\tilde M,\tilde M'}\in S^{\tilde M,\tilde M'}(\mathbb{R}_z^n; \mathbb{R}_{\tilde\sigma}^n \setminus 0, \mathbb{R}_{\tilde\tau}^1) \text{ and }  \abs{\tilde\tau}<c_2 \abs{\tilde\sigma}.
\end{split}
\end{equation*}
Then the composition
\begin{equation*}
 A \circ B \in I^{p+p'+\frac{1}{2}}(C_{\Sigma}).
\end{equation*}
In the above notation, the canonical relations $\Delta, C_{\Sigma}$ and $C_0$ are defined, respectively, as:
\begin{equation*}
 \Delta=\{(x, \xi, y, \eta) \vert x=y, \xi=\eta\in \mathbb{R}^n \setminus 0\} 
\end{equation*}
\begin{equation*}
C_{\Sigma}=\{(x, \xi, y, \eta) \vert x=y, \eta=\xi-\nabla h(x) t, x\in S, t\in \mathbb{R}^1, (\xi, t)\in \mathbb{R}^{n+1} \setminus 0\} 
\end{equation*}
\begin{equation*}
C_0=\{(x, 0, y, \nabla h(x) t) \vert x\in \mathbb{R}^n, y\in S, t\in \mathbb{R}^1\setminus 0\}. 
\end{equation*}
\end{theorem}

The proof of this theorem is given in Section \ref{sec-Composition}.

\subsection{Iterative construction of Green's function} \label{sec-Inversion}

This section gives the proof of Theorem {\ref{thm-chi1inversion}}.
First, we will start with the inversion of $A_1^2$ as defined in (\ref{eq-A12}).
The integral $A_1^2$ comes as a representation of $\gamma(x)\Laplace \in I^{\mu+\frac{5}{2}, -\mu-\frac{1}{2}}(\Delta, C_{\Sigma})$,
and the integral $A_1^1$ comes as a representation of $(\grad \gamma(x))\cdot\nabla \in I^{\mu+\frac{5}{2}, -\mu-\frac{3}{2}}(\Delta, C_{\Sigma})$ on $\chi_1$.
For the conductivity, $\gamma(x) \in I^{\mu}(S)$, we can write:
\begin{equation}\label{eq-gamma-g0-g1} 
 \gamma(x)=\gamma_0(x)+\gamma_1^{\pm}(x)h(x)_{\pm}^{-\mu-1}=\gamma_0(x)\cdot[1+\frac{\gamma_1^{\pm}(x)}{\gamma_0(x)}h(x)_{\pm}^{-\mu-1}], 
\end{equation}
where $\gamma_0$ is a smooth background, and $\gamma_1^+$ and $\gamma_1^-$ are smooth functions on either side of the hypersurface $S$. Since $\mu<-1$, it means $-\mu-1>0$, 
which means $\gamma(x)$ is continuous across the hypersurface $S$ with $\gamma=\gamma_0$ on $S$, but it can have a jump in derivative since $\gamma_1^+(x)$ may differ 
from $\gamma_1^-(x)$ as $x \rightarrow S$.
We look for $B_1^1$, such that its symbol is the reciprocal of the principal symbol of $A_1^2$, namely:
\begin{equation}\label{eq-B11-sigma}
\sigma(B_1^1)=\sigma(A_1^2)^{-1}=\abs{\sigma}^{-2}\frac{1}{\gamma_0(x)}\left[1-\frac{\gamma_1^{\pm}(x)}{\gamma_0(x)}h(x)_{\pm}^{-\mu-1} + \left(\frac{\gamma_1^{\pm}(x)}{\gamma_0(x)}h(x)_{\pm}^{-\mu-1}\right)^2-...\right].
 \end{equation}
In order to identify the $I^{p,l}$ class of $B_1^1$, we use the following identities:
\begin{equation*}
\begin{array}{l}
p+l=-2\\
l-\frac{1}{2}=-\mu-1.
\end{array} 
\end{equation*}

Solving this system for $p$ and $l$ lets us conclude that $B_1^1\in I^{\mu-\frac{3}{2}, -\mu-\frac{1}{2}}(\Delta, C_{\Sigma})$, as claimed in the theorem, when $i=1$ and $j=1$. 

From here on in this section, we will write for brevity simply $I^{p,l}$, when we mean the class $I^{p,l}(\Delta, C_{\Sigma})$. 

By Antoniano and Uhlmann's theorem, Theorem \ref{thm-Antoniano}, the composition:
\begin{equation} \label{eq-A12-B11}
 A_1^2 \circ B_1^1 \in I^{\mu+\frac{5}{2}, -\mu-\frac{1}{2}} \circ I^{\mu-\frac{3}{2}, -\mu-\frac{1}{2}}
 \subset I^{2\mu+\frac{3}{2}, -2\mu-\frac{3}{2}}.
 \end{equation}
Note that the composition is in a class that includes the identity operator Id, since the sum of the two superscripts is $0$, therefore including the pseudodifferential operators of order $0$. Moreover, its principal symbol is 1, by construction. Denoting the residual class terms as $E_1^{1,1}$ and $E_1^{1,2}$, respectively, we can write
\begin{equation}\label{eq-sum2errors}
 A_1^2\circ B_1^1-Id = E_1^{1,1}+E_1^{1,2}\in I^{2\mu+\frac{1}{2}, -2\mu-\frac{3}{2}}+I^{2\mu+\frac{3}{2}, -2\mu-\frac{5}{2}}
\end{equation}
Now, we denote by $E_1^{1,3}$ the composition $A_1^1 \circ B_1^1$, which we can once again understand by utilizing Theorem \ref{thm-Antoniano}:
\begin{equation} \label{eq-E113}
 E_1^{1,3}=A_1^1 \circ B_1^1\in I^{\mu+\frac{5}{2}, -\mu-\frac{3}{2}} \circ I^{\mu-\frac{3}{2}, -\mu-\frac{1}{2}} \subset I^{2\mu+\frac{3}{2}, -2\mu-\frac{5}{2}}.
\end{equation}
Notice that $E_1^{1,3}$ is in the same class as $E_1^{1,2}$ and 
\begin{equation} \label{eq-sum3errors}
 (A_1^2+A_1^1)\circ B_1^1 - Id = E_1^{1,1}+(E_1^{1,2}+E_1^{1,3}).
\end{equation}
We proceed with the construction of $B_1^2$, such that 
\begin{equation}
 \sigma(A_1^2 \circ B_1^2)=-\sigma(E_1^{1,2}+E_1^{1,3}) \in I^{2\mu+\frac{3}{2}, -2\mu-\frac{5}{2}}.
\end{equation}
By the requirements above and Theorem \ref{thm-Antoniano}, we can conclude that $B_1^2\in I^{\mu-\frac{3}{2}, -\mu-\frac{3}{2}}$ and
\begin{equation}
 A_1^2\circ B_1^2-(E_1^{1,2}+E_1^{1,3})=0 \text{ mod } I^{2\mu+\frac{1}{2}, -2\mu-\frac{5}{2}}+I^{2\mu+\frac{3}{2}, -2\mu-\frac{7}{2}}:=E_1^{2,1}+E_1^{2,2}.
\end{equation}
Now, the composition $E_1^{2,3}:=A_1^1\circ B_1^2$ matches $E_1^{2,2}$, just like in the previous stage, and we have:
\begin{equation}
 (A_1^2+A_1^1) \circ (B_1^1 + B_1^2) - Id = E_1^{1,1} + E_1^{2,1}+(E_1^{2,2}+E_1^{2,3}).
\end{equation}
We continue analogously with the construction of each consecutive $B_1^j$, $j \in N$. In particular, $B_1^{j+1}$ is constructed, so that
\begin{equation}
 \sigma(A_1^2 \circ B_1^{j+1})=-\sigma (E_1^{j,2}+E_1^{j,3})
\end{equation}
and as a consequence,
\begin{equation}
 (A_1^2+A_1^1) \circ (\sum_{j=1}^N B_1^j)-Id=\sum_{j=1}^N E_1^{j,1}+(E_1^{N,2}+E_1^{N,3}).
\end{equation}
At the end of stage $i=1$, we can express:
\begin{equation}
 B_1\sim \sum B_1^j \in I^{\mu-\frac{3}{2}, -\mu-\frac{1}{2}}
\end{equation}
and
\begin{equation}
 E_1\sim \sum E_1^{j,1} \in I^{2\mu+\frac{1}{2}, -2\mu-\frac{3}{2}}.
\end{equation}
If we now analyze the expression $(A_2+A_1)\circ B_1-Id-E_1$, we can conclude that
\begin{equation}
 (A_2+A_1)\circ B_1 - Id = E_1 \in I^{2\mu+\frac{1}{2}, -2\mu-\frac{3}{2}} \text{ mod } I^{2\mu+\frac{3}{2}}(C_{\Sigma}),
\end{equation}
by the following analysis. For brevity of notation, we omit the $\circ$ sign, but composition should be understood. For an arbitrarily large $N \in \mathbb{N}$, we can write:
\begin{equation*}
 \begin{split}
  &(A_1^2+A_1^1)B_1-Id-E_1\\
  &=(A_1^2+A_1^1)\left((B_1-\sum_{j=1}^N B_1^j)+\sum_{j=1}^N B_1^j\right)-Id-\left((E_1-\sum_{j=1}^N E_1^{j,1})+\sum_{j=1}^N E_1^{j,1}\right)\\
  &=(A_1^2+A_1^1)(B_1-\sum_{j=1}^N B_1^j)+\left((A_1^2+A_1^1)\sum_{j=1}^N B_1^j-Id-\sum_{j=1}^N E_1^{j,1}\right)-(E_1-\sum_{j=1}^N E_1^{j,1})\\
  &= (A_1^2+A_1^1)(B_1-\sum_{j=1}^N B_1^j)+(E_1^{N,2}+E_1^{N,3})-(E_1-\sum_{j=1}^N E_1^{j,1})\\
  & \in  I^{\mu+\frac{5}{2},-\mu-\frac{1}{2}} \circ I^{\mu-\frac{3}{2},-\mu-N-\frac{1}{2}} +
     I^{2\mu+\frac{3}{2},-2\mu-\frac{3}{2}-N}
    -  I^{2\mu+\frac{1}{2},-2\mu-\frac{3}{2}-N}\\
  & \subset I^{2\mu+\frac{3}{2},-2\mu-\frac{3}{2}-N}+
     I^{2\mu+\frac{3}{2},-2\mu-\frac{3}{2}-N}
    -  I^{2\mu+\frac{1}{2},-2\mu-\frac{3}{2}-N}\\
  & \subset  I^{2\mu+\frac{3}{2},-2\mu-\frac{3}{2}-N}
 \end{split}
\end{equation*}
The $I^{p,l}$ classes are nested, i.e., for $p'<p\in \mathbb{R}$ or $l'<l \in \mathbb{R}$, $I^{p',l}\subset I^{p,l}$ and $I^{p,l'}\subset I^{p,l}$. To obtain the last inclusion, we have thus used $I^{2\mu+\frac{1}{2},-2\mu-\frac{3}{2}-N} \subset I^{2\mu+\frac{3}{2},-2\mu-\frac{3}{2}-N}$.
This is valid for every $N$, so we can take the limit as $N \rightarrow \infty$.  Using the fact stated in \ref{prop-ipl-caps}, that 
$$\bigcap_{l \in \mathbb{R}} I^{p,l}(\Delta, C_{\Sigma})=I^p(C_{\Sigma}),$$
we can write
$$(A_1^2+A_1^1)B_1-Id = E_1 \in I^{2\mu+\frac{1}{2},-2\mu-\frac{3}{2}} (\Delta, C_{\Sigma}) \text{ mod } I^{2\mu+\frac{3}{2}}(C_{\Sigma}).$$
For later derivations, we label the residual above as $M_1\in I^{2\mu+\frac{3}{2}}(C_{\Sigma})$. 
Using the mapping properties of FIO's, we know that $I^{2\mu+\frac{3}{2}}(C_{\Sigma}):H^s \rightarrow H^{s-(2\mu+\frac{3}{2})-\frac{1}{2}}$, or $H^s \rightarrow H^{s-(2\mu+2)}$, which is a smoothing operator as long as $2\mu+2<0$, or $\mu<-1$, which has been assumed all along.

Now, we begin the second stage of the iteration process and we indicate that by using the subscript $i=2$ for all operators in this stage. We construct $B_2^1$, such that
$\sigma(A_1^2\circ B_2^1)=-\sigma(E_1)$. Since $A_1^2 \in I^{\mu+\frac{5}{2}, -\mu-\frac{1}{2}}$ and $E_1 \in I^{2\mu+\frac{1}{2},-2\mu-\frac{3}{2}}$, by Theorem \ref{thm-Antoniano}, we can deduce that
\begin{equation*}
B_2^1 \in I^{\mu-\frac{5}{2}, -\mu-\frac{1}{2}}.
\end{equation*}
Consequently, 
\begin{equation*}
(A_1^2+A_1^1)\circ (B_1+B_2^1)-Id \in I^{2\mu-\frac{1}{2},-2\mu-\frac{3}{2}}+I^{2\mu+\frac{1}{2},-2\mu-\frac{5}{2}}:=E_2^{1,1}+E_2^{1,2} \text{ mod } I^{2\mu+\frac{3}{2}}(C_{\Sigma}).
\end{equation*}
Then 
$$A_1^1\circ B_2^1 \in I^{\mu+\frac{5}{2}, -\mu-\frac{3}{2}} \circ I^{\mu-\frac{5}{2}, -\mu-\frac{1}{2}}:=E_2^{1,3}\in I^{2\mu+\frac{1}{2}, -2\mu-\frac{5}{2}}.$$
Just as in the previous stage, we have that $E_2^{1,2}$ and $E_2^{1,3}$ belong to the same class, and the principal symbol of their sum guides us in the construction of $B_2^2$ such that:
\begin{equation*}
 \sigma(A_1^2 \circ B_2^2)=-\sigma(E_2^{1,2}+E_2^{1,3}).
\end{equation*}
Eventually, for stage 2, we have:
\begin{equation*}
 (A_1^2+A_1^1)\circ(B_1+\sum_{j=1}^N B_2^j)-Id=\sum_{j=1}^N E_2^{j,1}+(E_2^{N,2}+E_2^{N,3})+M_1,\quad M_1\in I^{2\mu+\frac{3}{2}}(C_{\Sigma}). 
\end{equation*}
At the end of stage $i=2$, we can express
\begin{equation*}
 B_2\sim \sum_{j=1}^{\infty} B_2^j \in I^{\mu-\frac{5}{2},-\mu-\frac{1}{2}}
\end{equation*}
and
\begin{equation*}
 E_2\sim \sum_{j=1}^{\infty} E_2^{j,1} \in I^{2\mu-\frac{1}{2},-2\mu-\frac{3}{2}}.
\end{equation*}

If we now analyze the expression $(A_2+A_1)\circ (B_1+B_2)-Id-E_2$, we can conclude that
\begin{equation}
 (A_2+A_1)\circ (B_1+B_2) - Id = E_2 \in I^{2\mu-\frac{1}{2}, -2\mu-\frac{3}{2}} \text{ mod } I^{2\mu+\frac{3}{2}}(C_{\Sigma}),
\end{equation}
using the following analysis. For brevity of notation, we omit the $\circ$ sign, but compositions should be understood. For an arbitrarily large $N \in \mathbb{N}$, we can write:
\begin{equation*}
 \begin{split}
  &(A_1^2+A_1^1)(B_1+B_2)-Id-E_2\\
  =&(A_1^2+A_1^1)\left((B_1+\sum_{j=1}^N B_2^j)+(B_2-\sum_{j=1}^N B_2^j)\right)-Id-\left((E_2-\sum_{j=1}^N E_2^{j,1})+\sum_{j=1}^N E_2^{j,1}\right)\\
  =&\left((A_1^2+A_1^1)(B_1+\sum_{j=1}^N B_2^j)-Id-\sum_{j=1}^N E_2^{j,1}\right)
  +(A_1^2+A_1^1)(B_2-\sum_{j=1}^N B_2^j)-(E_2-\sum_{j=1}^N E_2^{j,1})\\
  =&E_2^{N,2}+E_2^{N,3}+M_1+(A_1^2+A_1^1)(B_2-\sum_{j=1}^N B_2^j)-(E_2-\sum_{j=1}^N E_2^{j,1})\\
  \in &  I^{2\mu+\frac{1}{2},-2\mu-\frac{3}{2}-N} + I^{2\mu+\frac{3}{2}}(C_{\Sigma})+
  I^{\mu+\frac{5}{2},-\mu-\frac{1}{2}} \circ I^{\mu-\frac{5}{2},-\mu-N-\frac{1}{2}} 
     -  I^{2\mu-\frac{1}{2},-2\mu-\frac{3}{2}-N}\\
 \subset & I^{2\mu+\frac{1}{2},-2\mu-\frac{3}{2}-N}+I^{2\mu+\frac{3}{2}}(C_{\Sigma})+
     I^{2\mu+\frac{1}{2},-2\mu-\frac{3}{2}-N}
    -  I^{2\mu-\frac{1}{2},-2\mu-\frac{3}{2}-N}\\
\subset & I^{2\mu+\frac{1}{2},-2\mu-\frac{3}{2}-N}+I^{2\mu+\frac{3}{2}}(C_{\Sigma})
 \end{split}
\end{equation*}
This is valid for every $N$, so we can take the limit as $N \rightarrow \infty$. The first term, which we can label with $M_2$, approaches $I^{2\mu+\frac{1}{2}}(C_{\Sigma})\hookrightarrow I^{2\mu+\frac{3}{2}}(C_{\Sigma})$.
As a conclusion, we can write that after the second stage of the inversion:
$$(A_1^2+A_1^1)(B_1+B_2)-Id - E_2=M_1+M_2 \in I^{2\mu+\frac{3}{2}}(C_{\Sigma}) + I^{2\mu+\frac{1}{2}}(C_{\Sigma})\subset I^{2\mu+\frac{3}{2}}(C_{\Sigma}),$$ where the error $E_2 \in I^{2\mu-\frac{1}{2}, -2\mu-\frac{3}{2}}$ 
is the initial input for the construction of the first operator $B_3^1$ in stage 3, namely, $\sigma(A_1^2\circ B_3^1)=-\sigma(E_2)$.    

As the stages progress, we construct $B_i=\sum_{j=1}^{\infty} B_i^j \in I^{\mu-i-\frac{1}{2},-\mu-\frac{1}{2}}(\Delta, C_{\Sigma})$, and we can build the asymptotic sum 
$B \sim \sum_{i=1}^{\infty} B_i \in I^{\mu-\frac{3}{2},-\mu-\frac{1}{2}}(\Delta, C_{\Sigma})$. $B$ is the desired Green's function. We claim that it serves as a right inverse to $A_1^2+A_1^1$ up to an operator in $I^{2\mu+\frac{3}{2}}(C_{\Sigma})$. This can be seen as follows:
\begin{equation*}
 \begin{split}
  &(A_1^2+A_1^1)B-Id\\
  =&(A_1^2+A_1^1)\left(\sum_{i=1}^N B_i+(B-\sum_{i=1}^N B_i)\right)-E_N+E_N-Id\\
  =&(A_1^2+A_1^1)\sum_{i=1}^N B_i-Id-E_N + (A_1^2+A_1^1)(B-\sum_{i=1}^N B_i) + E_N \\
   =&\sum_{i=1}^N M_i + (A_1^2+A_1^1)(B-\sum_{i=1}^N B_i) + E_N \\
 \in &  I^{2\mu+\frac{3}{2}}(C_{\Sigma})+
  I^{\mu+\frac{5}{2},-\mu-\frac{1}{2}} \circ I^{\mu-N-\frac{3}{2},-\mu-\frac{1}{2}} 
    +  I^{2\mu+\frac{3}{2}-N,-2\mu-\frac{3}{2}}\\
 \subset &  I^{2\mu+\frac{3}{2}}(C_{\Sigma})+  I^{2\mu+\frac{3}{2}-N,-2\mu-\frac{3}{2}}. 
 \end{split}
\end{equation*}
The above is valid for every $N$, and we can take the limit as $N\rightarrow \infty$. We use the fact that 
$$\cap_{p} I^{p,l}(\Delta, C_{\Sigma})=C^{\infty},$$ and the second summand disappears, leaving us with
$$(A_1^2+A_1^1)B-Id \in I^{2\mu+\frac{3}{2}}(C_{\Sigma}),$$ a smoothing operator, as already discussed.
This concludes the proof of the theorem. 

\subsection{Composition calculus: $I^{p,l}(C_0, C_{\Sigma})\circ I^{p',l'}(\Delta, C_{\Sigma})$} \label{sec-Composition}

The goal of this section is to understand the remaining compositions $A_2 \circ B$, $A_3^{\mu} \circ B$ and $A_3^{\mu+1} \circ B$. For this purpose, it is enough to prove Theorem 
\ref{thm-composition}. We will write in detail the case $A \in I^{p,l}(C_0, C_{\Sigma})$. It will be easy to see that all the techniques and conclusions then will hold for $A \in I^p(C_{\Sigma})$.

Suppose that $A$ and $B$ are as outlined in Theorem \ref{thm-composition}, namely, for constants
$c_1<c_2<\frac{1}{2}<2<c_3<c_4<\infty$, assume the kernel of $A$ has the representation:
\begin{equation}\label{eq-A}
\begin{split}
& K_A(x, z)=\int e^{i[(x-z)\cdot \sigma + h(z) \tau]} a_{M,M'}(z; \tau, \sigma) \, d\sigma \, d\tau, \\
& a_{M,M'}\in S^{M,M'}(\mathbb{R}_z^n; \mathbb{R}_{\tau}^1\setminus 0, \mathbb{R}_{\sigma}^n); c_3\abs{\sigma}<\abs{\tau};
\end{split}
\end{equation}
and the kernel of $B$ has the representation:
\begin{equation}\label{eq-B}
\begin{split}
&K_B(z, y)=\int e^{i[(z-y)\cdot \tilde\sigma - h(z) \tilde\tau]} b_{\tilde M,\tilde M'}(z; \tilde\sigma, \tilde\tau) \, d\tilde\sigma \, d\tilde\tau, \\
&b_{\tilde M,\tilde M'}\in S^{\tilde M,\tilde M'}(\mathbb{R}_z^n; \mathbb{R}_{\tilde\sigma}^n \setminus 0, \mathbb{R}_{\tilde\tau}^1); \abs{\tilde\tau}<c_2 \abs{\tilde\sigma}.
\end{split}
\end{equation}
Then the kernel of the composition $A \circ B$ can be represented as the integral
\begin{equation}\label{eq-AB}
\begin{split}
 K_{A \circ B} &= \int_{\mathbb{R}^{n}} A(x,z)B(z,y) \,dz\\
 &=\int_{\mathbb{R}^{3n+2}} e^{i[(x-z)\cdot \sigma + (z-y)\cdot \tilde\sigma + h(z)(\tau-\tilde\tau)]} a_{M,M'}(z; \tau, \sigma)b_{\tilde M,\tilde M'}(z; \tilde\sigma, \tilde\tau) \,dz \, d\sigma \, d\tau d\tilde\sigma \, d\tilde\tau
 \end{split}
\end{equation}
First, we note that away from the region, where $\abs{\tilde\sigma}\sim\abs{\tau}$, the integral is a $C^{\infty}$ function, which can be proved via integration by parts. We will
introduce a partition of unity $\psi_1+\psi_2+\psi_3\equiv 1$, that will allow us to focus on the regions, where $\abs{\tau}\lesssim \abs{\tilde \sigma}$, where $\abs{\tau}\sim \abs{\tilde \sigma}$, and where $\abs{\tilde \sigma}\lesssim \abs{\tau}$. Let:
\begin{equation} \label{eq-eps2}
 \epsilon_2=\frac{(1-c_2)-\frac{1}{2}}{1+\frac{1}{c_3}}
\end{equation}
and 
\begin{equation} \label{eq-eps3}
 \epsilon_3=\frac{(1+c_2)}{(1-\frac{1}{c_3})-\frac{1}{2}}.
\end{equation}
With these choices, recalling that $0<c_1<c_2<\frac{1}{2}<2<c_3<c_4<\infty$, it is clear that $0<\epsilon_2<1$ and $1<\epsilon_3<\infty$. Now make choices $\epsilon_1$ and $\epsilon_4$, such that:
\begin{equation}\label{eq-epsilons}
 0<\epsilon_1<\epsilon_2<1<\epsilon_3<\epsilon_4<\infty.
\end{equation}
Let's introduce the open cover $V$ on $S^1$, with:
\begin{equation*}
 V=\Big\{\{\frac{\abs{\tau}}{\abs{\tilde \sigma}}<\epsilon_2\}, \{\epsilon_1<\frac{\abs{\tau}}{\abs{\tilde \sigma}}<\epsilon_4\}, \{\epsilon_3<\frac{\abs{\tau}}{\abs{\tilde \sigma}}\}\Big\}.
\end{equation*}
Let  $\psi_1+\psi_2+\psi_3\equiv 1$ be a partition of unity on $S^1$, subordinate to this open cover. Extend them as homogeneous functions of degree 0 to all of $\mathbb{R}^2\setminus 0$, and then to $\mathbb{R}^{n+1}\setminus 0$, with $\tilde \psi_i(\tau, \tilde\sigma)=\psi_i(\frac{\abs{\tau}}{\abs{\tilde \sigma}})$.
In order to do the integration by parts we consider: 
\begin{equation}
 \abs{\Phi_z}=\abs{-\sigma + \tilde\sigma +\nabla h(z)(\tau - \tilde \tau)}=\abs{-\sigma + \nabla h(z) \tau + \tilde \sigma - \nabla h(z) \tilde \tau}, 
\end{equation}
where
\begin{equation}
 \Phi=(x-z)\cdot \sigma + (z-y)\cdot \tilde\sigma + h(z)(\tau-\tilde\tau)
\end{equation}
is the phase function of the integral. 

Note that on the support of $\psi_1$, $\abs{\tau} < \epsilon_2 \abs{\tilde \sigma}$, and we have:
\begin{equation}
 \begin{split}
  \abs{\Phi_z}&=\abs{-\sigma + \nabla h(z) \tau + \tilde \sigma - \nabla h(z) \tilde \tau}\\
  &\geq \abs{\tilde \sigma - \nabla h(z) \tilde \tau}-\abs{-\sigma + \nabla h(z) \tau}\\
  &\geq \abs{\tilde \sigma}-\abs{\tilde \tau} - (\abs{\sigma}+ \abs{\tau})\\
  &>(1-c_2)\abs{\tilde\sigma}-(1+\frac{1}{c_3})\abs{\tau}\\
  &>[(1-c_2) - (1+\frac{1}{c_3})\epsilon_2] \abs{\tilde\sigma}\\
  &=\frac{1}{2}\abs{\tilde \sigma}
  \end{split}
\end{equation}
Since on $\psi_1$, $\abs{\nabla_z \Phi}>\frac{1}{2}\abs{\tilde \sigma}$, we can conclude there exists a first order $\Psi DO$, $L_{\tilde \sigma}=O(\frac{1}{\tilde \sigma})\nabla_z$, such that $L_{\tilde \sigma}^N(e^{i\Phi})=e^{i\Phi}$, for all $N\in \mathbb{N}$. Then on $\psi_1$, integrating by parts gives:
\begin{equation}
\begin{split}
 K_{A \circ B} (x,y) = \pm & \int_{\mathbb{R}^n} \int_{\abs{\tilde \tau}<c_2\abs{\tilde \sigma}} \int_{\abs{\tau}<\epsilon_2 \abs{\tilde\sigma}} \int_{\abs{\sigma}<\frac{1}{c_3}\abs{\tau}}\int_{z \in C}\\
      &e^{i\Phi}(L_{\tilde \sigma}^t)^N (a_{M,M'}(z; \tau, \sigma)b_{\tilde M,\tilde M'}(z; \tilde\sigma, \tilde\tau)) \,dz \, d\sigma \, d\tau \, d\tilde\tau \, d\tilde\sigma.
\end{split}
\end{equation}
In the above integral notation, $C$ stands for a compact set. Provided $N$ is large enough, the order of $\tilde \sigma$ can be made arbitrarily negative, which will eventually ensure the absolute convergence of the integral. Derivatives in $x$ and $y$ just increase the effective orders of $\sigma$ and $\tilde \sigma$ in the amplitude, but that can be offset by increasing $N$. Thus, $K_{A \circ B}$ is $C^{\infty}$ on $\psi_1$.

In a very similar fashion, we prove that on the support of $\psi_3$, where $\abs{\tilde \sigma}<\frac{1}{\epsilon_3 }\abs{\tau}$, $K_{A \circ B}$ is $C^{\infty}$. 
Note that on $\psi_3$,
\begin{equation}
 \begin{split}
  \abs{\Phi_z}&\geq\abs{-\sigma + \nabla h(z) \tau} -\abs{\tilde \sigma - \nabla h(z) \tilde \tau}\\
   &\geq \abs{\tau} -\abs{\sigma} - (\abs{\tilde \sigma}+\abs{\tilde \tau})\\
  &>(1-\frac{1}{c_3})\abs{\tau}-(1+c_2)\abs{\tilde\sigma}\\
  &>[(1-\frac{1}{c_3})-(1+c_2)\frac{1}{\epsilon_3}] \abs{\tau}\\
  &=\frac{1}{2}\abs{\tau}
  \end{split}
\end{equation}

Since on $\psi_3$, $\abs{\nabla_z \Phi}>\frac{1}{2}\abs{\tau}$, we can conclude there exists a first order $\Psi DO$, $L_{\tau}=O(\frac{1}{\tau})\nabla_z$, such that $L_{\tau}^N(e^{i\Phi})=e^{i\Phi}$, for all $N\in \mathbb{N}$. Then on the support of $\psi_3$, integrating by parts gives:
\begin{equation}
\begin{split}
 K_{A \circ B (x,y)} = \pm & \int_{\mathbb{R}^n} \int_{\abs{\sigma}<\frac{1}{c_3} \abs{\tau}} \int_{\abs{\tilde\sigma}<\frac{1}{\epsilon_3}\abs{\tau}} \int_{\abs{\tilde\tau}<c_2\abs{\tilde \sigma}} \int_{z \in C} \\
      & e^{i\Phi}(L_{\tau}^t)^N \left(a_{M,M'}(z; \tau, \sigma)b_{\tilde M,\tilde M'}(z; \tilde\sigma, \tilde\tau)\right) \,dz \, d\tilde\tau \, d\tilde\sigma \, d\sigma \, d\tau.
\end{split}
\end{equation}
This is absolutely integrable, provided $N$ is chosen to be sufficiently large. For derivatives in $x$ and $y$, we simply need to increase $N$ to ensure absolute integrability. 
Thus, $A \circ B$ is
$C^{\infty}$ on the support of $\psi_3$.

Notice that even if $A \in I^p(C_{\Sigma})$, similar integration by parts leads to the same conclusion, namely $A \circ B$ is $C^{\infty}(\mathbb{R}^{2n})$ on the supports of $\psi_1$ and $\psi_3$.

Now, we can use iterated regularity to prove that on $\psi_2$, where $\abs{\tilde \sigma}\sim \abs{\tau}$, the integral (\ref{eq-AB}) is in $I^{p''}(C_{\Sigma})$ for some $p''$.
First, we look at the generating functions for the canonical relation:
$$C_{\Sigma}=\{(x, \xi, y, \eta) \vert x=y, \eta=\xi-\nabla h(x) t, x\in S, t\in \mathbb{R}^1, (\xi, t)\in \mathbb{R}^{n+1} \setminus 0\}.$$

We have that the following set of functions vanish on $C_{\Sigma}$ and form a redundant set of defining functions:
\begin{equation*}
 \begin{split}
 & x_j-y_j, 1\leq j \leq n \\
 & h(x) \\
 & (\xi - \eta) \wedge \nabla h(x).
 \end{split}
\end{equation*}
The last equation means that any $2 \times 2$ determinants of the $2 \times n$ matrix must be singular, i.e., $\frac{\partial h_j}{\partial x_j}(\xi_i-\eta_i)-\frac{\partial h_i}{\partial x_i}(\xi_j-\eta_j))=0, i \neq j$. To get the ideal generating the Lagrangian:
$$\Lambda = C_{\Sigma}'=\{(x, \xi, y, -\eta) \vert x=y, \eta=\xi-\nabla h(x) t, x\in S, t\in \mathbb{R}^1, (\xi, t)\in \mathbb{R}^{n+1} \setminus 0\}, $$
simply replace $\eta$ with $-\eta$, i.e., the set of functions
\begin{equation}
 \begin{split}
 & h(x) \\
 & x_j-y_j, 1\leq j \leq n \\
 & \frac{\partial h_j}{\partial x_j}(\xi_i+\eta_i)-\frac{\partial h_i}{\partial x_i}(\xi_j+\eta_j), i\neq j, 1\leq i, j \leq n 
 \end{split}
\end{equation}
vanish on $\Lambda$.
The first order $\Psi DO$s vanishing on $\Lambda$, are all in the form:
\begin{equation}
 \alpha_0 h(x) \abs{\xi}+\sum_{j=1}^n \beta_0^j (x_j-y_j) \abs{\xi}+\sum_{i,j=1}^n \gamma_0^{i,j}\left(\frac{\partial h_j}{\partial x_j}(\xi_i+\eta_i)-\frac{\partial h_i}{\partial x_i}(\xi_j+\eta_j)\right)
\end{equation}
where $\alpha_0$, $\beta_0^j$, and $\gamma_0^{i,j}$ are smooth symbols of order 0. To execute the iterated regularity, we need to show that each of the ideal generating first order $\Psi DO$s with symbols:

\begin{equation}
 h(x) \abs{\xi},
\end{equation}
\begin{equation}
(x_j-y_j)\abs{\xi}, 1\leq j \leq n 
\end{equation}
and 
\begin{equation}
\frac{\partial h_j}{\partial x_j}(\xi_i+\eta_i)-\frac{\partial h_i}{\partial x_i}(\xi_j+\eta_j), i\neq j, 1\leq i, j \leq n,
\end{equation}
preserves the Sobolev space $H^s$, to which $A\circ B$ belongs.

Note that in general, see \cite{Duistermaat96}, if
\begin{equation*}
 u(x)=\int e^{i\phi(x, \theta)}a(x, \theta) d\theta,
\end{equation*}
then 
\begin{equation*}
 P(x, D) u(x)=\int e^{i\Phi(x, \theta)}\left(p(x, d_x \Phi(x, \theta))+\text{lower order terms}\right)a(x, \theta) d\theta.
\end{equation*}
Here
\begin{equation}
 \Phi = (x-z)\cdot \sigma + (z-y) \cdot \tilde\sigma + h(z) (\tau-\tilde\tau).
\end{equation}

For the generator $h(x)\abs{\xi}$ it is enough to check for symbols $p_1(x,y,\xi, \eta)=h(x)\xi_i$, since 
\begin{equation*}
 h(x)\abs{\xi}=h(x)\frac{\abs{\xi}^2}{\abs{\xi}}=h(x) \frac{\sum_{i=1}^n \xi_i^2}{\abs{\xi}}=h(x)\sum_{i=1}^n \frac{\xi_i}{\abs{\xi}} \xi_i,
\end{equation*}
and each $\frac{\xi_i}{\abs{\xi}}$ is a symbol of order $0$.
Note that $\frac{\partial}{\partial x_i} \Phi = \sigma_i$, so we need to check that 
\begin{equation} \label{eq-p1}
\begin{split}
 &P_1(x,y, D)K_{A\circ B}(x,y)\\
 &= \int e^{i\Phi} h(x) \sigma_i a_{M,M'}(\tau, \sigma) b_{\tilde M \tilde M'}(\tilde \sigma, \tilde \tau) \, dz \,d\sigma \,d\tau \, d\tilde\sigma \,d\tilde\tau\\
 &=\int e^{i\Phi} \left(h(z)+(h(x)-h(z))\right) \sigma_i a_{M,M'}(\tau, \sigma) b_{\tilde M \tilde M'}(\tilde \sigma, \tilde \tau) \, dz \,d\sigma \,d\tau \, d\tilde\sigma \,d\tilde\tau
\end{split}
\end{equation}
preserves the Sobolev space.
Now, we can write 
$$e^{i\Phi}h(z)=\frac{\partial}{\partial \tau} e^{i\Phi}$$ 
and
$$e^{i\Phi}(h(x)-h(z))=e^{i\Phi}\left((x-z)H(x,z)\right)=H(x,z) \frac{\partial}{\partial \sigma} e^{i\Phi}.$$
In the last expression, $H(x, z)$ comes from Taylor's expansion and is therefore $C^{\infty}$ and harmless for the Sobolev space behavior of the integral. 
Thus, we can rewrite (\ref{eq-p1}) as
\begin{equation}
\begin{split}
&\int e^{i\Phi} \frac{\partial}{\partial \tau}\left(\sigma_i a_{M, M'}(\tau, \sigma)\right)b_{\tilde M, \tilde M'}(\tilde \sigma, \tilde \tau) \, dz \, d\sigma \, d\tau \,d\tilde\sigma \,d\tilde\tau \\ 
+&\int e^{i\Phi} \frac{\partial}{\partial \sigma}\left(H(x,z) a_{M, M'}(\tau, \sigma)\right)b_{\tilde M, \tilde M'}(\tilde \sigma, \tilde \tau) \, dz \, d\sigma \, d\tau \,d\tilde\sigma \,d\tilde\tau.
\end{split}
\end{equation}
Since $\abs{\tau}$ is the larger variable relative to $\abs{\sigma}$, the above amplitudes have behavior no worse than the original amplitudes and therefore the Sobolev space is preserved. 

Now, we look at the second generator with symbol $p_2(x,y,\xi, \eta)=(x_j-y_j)\xi_i$, $1 \leq i,j \leq n$. The action of the pseudodifferential operator is expressed as:
\begin{equation} \label{eq-p2}
 P_2(x, y, D) K_{A \circ B} (x,y) = \int e^{i\Phi} (x_j-y_j) \sigma_i a_{M,M'}(\tau, \sigma) b_{\tilde M \tilde M'}(\tilde \sigma, \tilde \tau) \, dz \,d\sigma \,d\tau \, d\tilde\sigma \,d\tilde\tau.
\end{equation}
Now, we represent:
\begin{equation*}
 e^{i\Phi}(x_j-y_j)=(x_j-z_j+z_j-y_j) e^{i\Phi}=(\frac{\partial}{\partial \sigma_j}+\frac{\partial}{\partial \tilde\sigma_j})e^{i\Phi}.
\end{equation*}
Thus, (\ref{eq-p2}) becomes:
\begin{equation}
 \begin{split}
&\int e^{i\Phi} \frac{\partial}{\partial \sigma_j}(\sigma_i a_{M, M'}(\tau, \sigma))b_{\tilde M, \tilde M'}(\tilde \sigma, \tilde \tau) \, dz \, d\sigma \, d\tau \,d\tilde\sigma \,d\tilde\tau \\ 
+&\int e^{i\Phi} (\sigma_i a_{M, M'}(\tau, \sigma)) \frac{\partial}{\partial \tilde \sigma_j} b_{\tilde M, \tilde M'}(\tilde \sigma, \tilde \tau) \, dz \, d\sigma \, d\tau \,d\tilde\sigma \,d\tilde\tau.
  \end{split}
\end{equation}
Clearly, the first summand has no impact on the Sobolev space $H^s$. For the second summand we have to rely on the support of $\psi_2$, where $\abs{\tilde\sigma}\sim \abs{\tau}$. The order of the new amplitude is one worse in $\langle \sigma \rangle$, but one better is $\langle \tilde \sigma \rangle$, but $\abs{\sigma}<\frac{1}{c_3}\abs{\tau}<\epsilon_4\abs{\tilde \sigma}$, which implies the Sobolev space cannot have become any worse, i.e., it is preserved. 

Finally, we analyze the last generator, with symbol
$$p_3(x,y,\xi, \eta)=\frac{\partial h}{\partial y_j}(\xi_i+\eta_i)-\frac{\partial h}{\partial y_i}(\xi_j+\eta_j), i\neq j, 1\leq i,j \leq n,$$
where we have interchanged $x$ and $y$, because on $\Lambda$, $x=y$. Then:
\begin{equation} \label{eq-p3}
\begin{split}
 &P_3(x, y, D) K_{A \circ B} (x,y) \\
 &= \int e^{i\Phi} \left(\frac{\partial h}{\partial y_j}(\sigma_i-\tilde\sigma_i)-\frac{\partial h}{\partial y_i}(\sigma_j-\tilde\sigma_j)\right) a_{M,M'}(\tau, \sigma) b_{\tilde M \tilde M'}(\tilde \sigma, \tilde \tau) \, dz \,d\sigma \,d\tau \, d\tilde\sigma \,d\tilde\tau.
\end{split}
\end{equation}
Writing partial derivatives with subscript notation, i.e., $h_i(z)$ for $\frac{\partial h}{\partial z_i}$, note that:
\begin{equation}
 \frac{\partial}{\partial z_i} e^{i\Phi}=(-\sigma_i+\tilde\sigma_i+h_i(z)(\tau-\tilde\tau)) e^{i\Phi},
\end{equation}
and so
\begin{equation} \label{eq-hjy}
 (\sigma_i-\tilde\sigma_i) e^{i\Phi} = \left(-\frac{\partial}{\partial z_i} + h_i(z)(\tau-\tilde\tau)\right) e^{i\Phi},
\end{equation}
\begin{equation}\label{eq-hiy}
 (\sigma_j-\tilde\sigma_j) e^{i\Phi} = \left(-\frac{\partial}{\partial z_j} + h_j(z)(\tau-\tilde\tau)\right) e^{i\Phi}.
\end{equation}
Now, using Taylor's expansion, we can write:
\begin{equation}
 h_j(y)=h_j(z)+h_j(y)-h_j(z)=h_j(z) + \sum_{k=1}^n H_{jk}(y_k-z_k),
\end{equation}
\begin{equation}
 h_i(y)=h_i(z)+h_i(y)-h_i(z)=h_i(z) + \sum_{k=1}^n H_{ik}(y_k-z_k),
\end{equation}
where $H_{ik}$ and $H_{jk}$ are smooth functions, depending on $y_k$ and $z_k$.

We rewrite our multiplier from equation (\ref{eq-p3}) together with the oscillatory term $e^{i\Phi}$ to prepare for integration by parts:
\begin{equation}
 \begin{split}
&\left(h_j(y)(\sigma_i-\tilde\sigma_i)-h_i(y)(\sigma_j-\tilde\sigma_j)\right)e^{i\Phi}\\
&\\
&=(h_j(z)+(h_j(y)-h_j(z)))(\sigma_i-\tilde\sigma_i)-(h_i(z)+(h_i(y)-h_i(z)))(\sigma_j-\tilde\sigma_j))e^{i\Phi}\\
&\\
&= (h_j(z)(\sigma_i-\tilde\sigma_i)-h_i(z)(\sigma_j-\tilde\sigma_j))e^{i\Phi}\\
&\quad+\sum_{k=1}^n H_{jk} (y_k-z_k)(\sigma_i-\tilde\sigma_i)e^{i\Phi}\\
&\quad+\sum_{k=1}^n H_{ik} (y_k-z_k)(\sigma_j-\tilde\sigma_j)e^{i\Phi}\\
&\\
&=\left(h_j(z)\left(-\frac{\partial}{\partial z_i} + h_i(z)(\tau-\tilde\tau)\right)-h_i(z)\left(-\frac{\partial}{\partial z_j} + h_j(z)(\tau-\tilde\tau)\right)\right) e^{i\Phi} \\
&\quad+ \sum_{k=1}^n H_{jk}(-\frac{\partial}{\partial \tilde\sigma_k})(\sigma_i-\tilde\sigma_i) e^{i\Phi} - H_{ji}e^{i\Phi} \\
&\quad+ \sum_{k=1}^n H_{ik}(-\frac{\partial}{\partial \tilde\sigma_k})(\sigma_j-\tilde\sigma_j) e^{i\Phi} - H_{ij}e^{i\Phi} \\
&\\
&= \left(- h_j(z) \frac{\partial}{\partial z_i} + h_i(z) \frac{\partial}{\partial z_j}\right)e^{i\Phi}\\
&\quad+ \sum_{k=1}^n H_{jk}(-\frac{\partial}{\partial \tilde\sigma_k})\left(-\frac{\partial}{\partial z_i} + h_i(z)(\tau-\tilde\tau)\right) e^{i\Phi} \\
&\quad+ \sum_{k=1}^n H_{ik}(-\frac{\partial}{\partial \tilde\sigma_k})\left(-\frac{\partial}{\partial z_j} + h_j(z)(\tau-\tilde\tau)\right) e^{i\Phi} \\
&\quad- H_{ji}e^{i\Phi}- H_{ij}e^{i\Phi}.
 \end{split}
\end{equation}


The adjoint of the first term in the sum only depends on the spatial variable $z$ and has no effect on the phase variables of the amplitude. 
Applying the adjoint of the second term to the amplitude, after absorbing the smooth term $H$ into the amplitude, can be written as:
\begin{equation*}
 \begin{split}
  &\left(\frac{\partial}{\partial z_i} - h_i(z)(\tau-\tilde\tau)\right)\left(\frac{\partial}{\partial \tilde\sigma_k} a_{M,M'}(\tau, \sigma)b_{\tilde M, \tilde M'} (\tilde \sigma, \tilde \tau)\right)\\
  =&\left(\frac{\partial}{\partial z_i} - h_i(z)\tau+h_i(z)\tilde\tau\right)\left(\frac{\partial}{\partial \tilde\sigma_k} a_{M,M'}(\tau, \sigma)b_{\tilde M, \tilde M'} (\tilde \sigma, \tilde \tau)\right).
 \end{split}
\end{equation*}
This also leads to no worse Sobolev space, because $\abs{\tilde \tau} \lesssim \abs{\tilde \sigma}$ and $\abs{\tilde \sigma} \sim \abs{\tau}$. 
The same argument holds for the third term, where the roles of $i$ and $j$ are interchanged. In the last term, the smooth functions $H$ can be considered as symbols of order $0$, and they can be absorbed into the amplitude, leading to integrals of the same type as the original integral.

This concludes the iterated regularity proof. Namely $A\circ B \in I^{p''}(C_{\Sigma})$, for some $p''$. A very similar iterated regularity argument gives the case of $A \in I^p(C_{\Sigma}).$

In order to obtain the order of $p''$, we make use of the fact that 
$$I^{p,l}(C, C_{\Sigma})\hookrightarrow I^p(C_{\Sigma} \setminus C),$$
regardless of whether the first canonical relation $C$ is $C_0$ as in $A$ or $\Delta$ as in $B$. In addition, we use the fact that 
$$I^p(C_{\Sigma})\circ I^{p'}(C_{\Sigma})=I^{p+p'+\frac{1}{2}}(C_{\Sigma}).$$
These two facts immediately lead to the conclusion that $p''=p+p'+\frac{1}{2}$, regardless of whether $A\in I^{p,l}(C_0, C_{\Sigma})$ or $A\in I^{p}(C_{\Sigma})$, whenever $B \in I^{p',l'}(\Delta, C_{\Sigma})$. This concludes the proof of Theorem \ref{thm-composition}.

Let's recall the operators to which we need to apply Theorem \ref{thm-composition}. We were concerned with the compositions $A_2 \circ B$, $A_3^{\mu} \circ B$, and $A_3^{\mu+1} \circ B$. We have already shown that these operators belong to the following classes: 
\begin{equation}
 A_2 \in I^{\mu+\frac{5}{2}}(C_{\Sigma}), 
\end{equation}
\begin{equation}
 A_3^{\mu} \in I^{\mu+\frac{5}{2}, -2-\frac{n}{2}}(C_0, C_{\Sigma}), 
\end{equation}
\begin{equation}
 A_3^{\mu+1} \in I^{\mu+\frac{5}{2}, -1-\frac{n}{2}}(C_0, C_{\Sigma}), 
\end{equation}
and 
\begin{equation}
 B \in I^{\mu-\frac{3}{2}, -\mu-\frac{1}{2}}(\Delta, C_{\Sigma}).
\end{equation}
This means that $p=\mu+\frac{5}{2}$ for all of the above $A$, and $p'=\mu-\frac{3}{2}$, and therefore $A\circ B \in I^{2\mu+\frac{3}{2}}(C_{\Sigma})$. This is indeed the exact same  smoothing class, up to which we saw inversion of $(A_1^2+A_1^1)$ is possible. Therefore, the iteratively constructed operator $B$ is indeed an approximate Green's function for $L_{\gamma}$.

\section{Composition of Fourier Integral Operators in the Presence of a Zero Section}\label{Ch-FIO}
The standard clean intersection calculus assumes canonical relations avoiding the zero section altogether, with the only exception being \cite{Greenleaf01}. The canonical relation, $C_{\Sigma}$, described below, does not avoid the zero section, and this section investigates the composition of the operators associated to $C_{\Sigma}$. We begin by describing the class of Fourier integral operators $I^m(C_{\Sigma})$. Let $M$ be $n$-dimensional smooth manifold and let $S=\{h(x)=0\}$, where $h$ is a defining function. Then 
$S$ is a smooth submanifold of $M$ of codimension 1. Let $\Sigma=T^*M\vert_S=\{(x, \xi) \in T^*M\setminus 0\} \vert h(x)=0\}$ and $C_{\Sigma}\subset T^*M \times T^*M$ its flowout relation, generated by the Hamiltonian vector field $H_h=-\grad h\cdot\frac{d}{d\xi}$. To be specific $C_{\Sigma}$ is the canonical relation
$$C_{\Sigma}=\{ (x, \xi, y, \eta) \vert x\in S, y=x, \eta=\xi - t \grad h(x), t\in \RR \}.$$
Notice that this canonical relation can intersect the zero section both on the left (when $\xi=0$) and on the right (when $\eta=0$), but not both at the same time.
An operator $A$ is in $I^m(C_{\Sigma})$ if its kernel has a representation:
\begin{equation}\label{eq-K_A}
K_A(x,y):=\int_{\RR^{n+1}} e^{i[(x-y)\cdot \xi + h(x)\theta]} a (x,y;(\xi, \theta))\,d\xi\,d\theta,
\end{equation}
where $a \in S_{1,0}^M (\RR^{2n}; \RR^{n+1})$ is a standard symbol. The order of the symbol $a$ is $M=m-(n+1)/2+2n/4=m-1/2.$

\bigskip
In this section we state and prove the following theorem: 

\begin{theorem} \label{thm-FIO-0sec}
 Let $A_j \in I^{m_j}(C_{\Sigma})$, $m_j<-\frac{1}{2}$, for $j=1,2$, where we define
 $$C_{\Sigma}=\{(x, \sigma, x, \sigma-\tilde\rho \grad h(x)) \vert x\in S, \sigma \in \RR^n, \tilde\rho \in \RR, (\sigma, \tilde\rho) \neq(0^n,0)\}$$
 Then:
 $$A_1^*A_2 \in I^{p,l_1}(C_0, C_{\Sigma})+I^{p,l_2}(C_0^t, C_{\Sigma}),$$
 where 
 $$C_0:=\{(x,0, y, \rho \grad h(y)) \vert x \in \RR^n, y \in S, \rho \in \RR\setminus 0\}$$
 and
$$C_0^t:=\{(x,\rho \grad h(x), y, 0) \vert x \in S, y \in \RR^n, \rho \in \RR\setminus 0 \}$$
$$p=m_1+m_2+\frac{1}{2}$$ 
$$l_j=-\frac{n+1}{2}-m_j.$$
\end{theorem}

\bigskip
For $A_1 \in I^{m_1}(C_{\Sigma})$, $A_2 \in I^{m_2}(C_{\Sigma})$ with $m_1, m_2 < -\frac{1}{2}$, both properly supported operators, we would like to analyze the composition kernel $K_{A_1^*A_2}$, where $^*$ denotes the adjoint operator.
\begin{equation}\label {eq-composition}
K_{A_1^*A_2}(x,y) = \int_{\RR^n} K_{A_1^*}(x,z)K_{A_2}(z,y)\,dz,
\end{equation}
where
\begin{equation}\label{eq-adjKernel}
 K_{A_1^*}(x,z)=\overline{K_{A_1}(z,x)}=\int_{\RR^{n+1}} e^{i[(x-z)\cdot \sigma - h(z)\tau]}\overline{a_1}(z,x;(\sigma, \tau))\,d\sigma\, d\tau.
\end{equation}

Plugging equations (\ref{eq-K_A}) and (\ref{eq-adjKernel}) into equation (\ref{eq-composition}), we obtain:
\begin{equation}
\begin{split}
 &K_{A_1^*A_2}(x,y) \\
 &= \int_{\RR^{3n+2}} e^{i[(x-z)\cdot \sigma + (z-y)\cdot \tilde{\sigma}+h(z)(\tilde{\tau}-\tau)]}\overline{a_1}(z,x; (\sigma, \tau))a_2(z,y; (\tilde{\sigma}, \tilde{\tau}))\,d\sigma\, d\tau\, d\tilde{\sigma}\, d\tilde{\tau} dz.
\end{split}
 \end{equation}
 
With the change of variables, $\rho = \tau$ and $\tilde{\rho}=\tilde{\tau}-\tau$, the above becomes:
\begin{equation}
  \begin{split}
 &K_{A_1^*A_2}(x,y)\\
 &= \int_{\RR^{3n+1}} e^{i[(x-z)\cdot \sigma + (z-y)\cdot \tilde{\sigma}+h(z)\tilde{\rho}]} \left(\int_{\RR} \overline{a_1}(z,x; (\sigma, \rho))a_2(z,y; (\tilde{\sigma}, \tilde{\rho}+\rho))d\rho \right) \,d\tilde{\rho}\, d\sigma\, d\tilde{\sigma}\, dz \\
 &=\int_{\RR^{3n+1}} e^{i[(x-z)\cdot \sigma + (z-y)\cdot \tilde{\sigma}+h(z)\tilde{\rho}]} a_3(x,y,z; \sigma, \tilde{\sigma}, \tilde{\rho}) \,d\tilde{\rho}\,d\sigma \,d\tilde{\sigma}\,dz, 
  \end{split}
 \end{equation}
 where $$a_3(x,y,z; \sigma, \tilde{\sigma}, \tilde{\rho}):=\int_{\RR} \overline{a_1}(z,x; (\sigma, \rho))a_2(z,y; (\tilde{\sigma}, \tilde{\rho}+\rho))\,d\rho.$$ 
 
 In order to understand the composition kernel, we first set out to understand $a_3$. With a change of variable and relabeling $a_1(z,x; (\sigma, \rho)):=\overline{a_1}(z,x; (\sigma, -\rho)))$, $a_3$ can be seen as a partial convolution of symbols,
 \begin{equation}\label{eq-a3}
 \displaystyle a_3(x,y,z; \sigma, \tilde{\sigma}, \tilde{\rho}):=\int_{\RR} {a_1}(z,x; (\sigma, \rho))a_2(z,y; (\tilde{\sigma}, \tilde{\rho}-\rho))\,d\rho.
 \end{equation}
 Here the symbols satisfy the estimates specified in the lemma below.   
 
 \begin{lemma}\label{lemma-a3}
  Let $a_3(x,y,z; \sigma, \tilde{\sigma}, \tilde{\rho}):=\int_{\RR} {a_1}(z,x; (\sigma, \rho))a_2(z,y; (\tilde{\sigma}, \tilde{\rho}-\rho))\,d\rho$ with $a_1$ and $a_2$ satisfying the symbol estimates:
$$\abs{\partial_{x,z}^{\gamma} \partial_{\sigma}^{\alpha}\partial_{\rho}^{\beta} a_1(z,x; (\sigma, \rho))} \leq C_{\alpha \beta \gamma} \langle \sigma, \rho  \rangle^{M-|\alpha|-\beta},$$  $$\abs{\partial_{y,z}^{\gamma} \partial_{\tilde\sigma}^{\alpha}\partial_{\tilde\rho}^{\beta} a_2(z, y; (\tilde\sigma, \tilde\rho))} \leq C_{\alpha \beta \gamma} \langle \tilde\sigma, \tilde\rho  \rangle^{\tilde{M}-|\alpha|-\beta},$$
  where $M, \tilde{M}<-1$. Then: 
  \begin{equation} \label{orig-est}
\displaystyle |a_3 (x,y,z; \sigma, \tilde{\sigma}, \tilde{\rho})| \lesssim 
\begin{cases}
  \displaystyle \langle \tilde{\rho}\rangle^{\tilde{M}}\langle \sigma\rangle^{M+1}+\langle \tilde{\rho}\rangle^{M}\langle \tilde\sigma\rangle^{\tilde M+1} & \text{on}\ \Gamma_1:=\{\abs{\tilde{\rho}} \gtrsim \max(\abs{\sigma},\abs{\tilde{\sigma}})\} \\
  \displaystyle \frac{\langle \sigma \rangle^{M+1} \langle \tilde{\sigma} \rangle ^{\tilde{M}+1}}{\langle \sigma, \tilde{\sigma} \rangle} & \text{on}\ \Gamma_2:=\{\abs{\tilde{\rho}} \lesssim \max(\abs{\sigma},\abs{\tilde{\sigma}})\} 
\end{cases}
\end{equation}

Moreover, $\abs{\partial_{x,y,z}^{\delta}\partial_{\sigma}^{\gamma}\partial_{\tilde\sigma}^{\beta}\partial_{\tilde\rho}^{\alpha}a_3} \leq C_{\alpha \beta \gamma \delta} \text{(original estimate)} \langle \sigma \rangle^{-\abs{\gamma}}\langle\tilde \sigma \rangle^{-\abs{\beta}}\langle \sigma ,\tilde \sigma, \tilde \rho \rangle^{-\alpha}$,
where (original estimate) refers to equation (\ref{orig-est}).
\end{lemma}

\begin{proof}

First, we divide the domain into 5 conic regions, based on the external parameters' relation to each other: 

I. $\abs{\sigma}, \abs{\tilde{\sigma}} \lesssim \abs{\tilde{\rho}}$

II. $\abs{\tilde{\rho}} \lesssim \abs{\sigma} \lesssim \abs{\tilde{\sigma}}$

III. $\abs{\tilde{\rho}} \lesssim \abs{\tilde{\sigma}} \lesssim \abs{\sigma}$

IV. $ \abs{\sigma} \lesssim \abs{\tilde{\rho}} \lesssim \abs{\tilde{\sigma}}$

V. $\abs{\tilde{\sigma}} \lesssim \abs{\tilde{\rho}} \lesssim \abs{\sigma}$

Throughout the text, we will use $\lesssim$ and $\gtrsim$ to mean comparison up to multiplicative constants, while $\sim$ means the ratio of the two quantities on the left and on the right is bounded below and above by a constant. For appropriate constants above, the intersections of these regions with the $2n$-dimensional sphere forms an open cover of the sphere. We introduce a partition of unity $\{\psi_j\}$ subordinate to this open cover and extend them as functions homogeneous of degree zero to the entire $2n+1$ dimensional space. The partition of unity then belongs to the standard symbol class $S^0$, and multiplication by the $\psi_j$ does not affect our symbol estimates.  

Without loss of generality, assume $\tilde\rho>0$. Since the integrating variable $\rho$ in (\ref{eq-a3}) is 1-dimensional, we divide the real line into four intervals. 

$K_1.$ $\abs{\rho}\geq \frac{3}{2}\tilde{\rho}$

$K_2.$ $-\frac{1}{2}{\tilde{\rho}}\leq \rho \leq \frac{1}{2}{\tilde{\rho}}$

$K_3.$ $\frac{1}{2}\tilde{\rho}\leq \rho \leq \frac{3}{2}\tilde{\rho}$ 

$K_4.$ $-\frac{3}{2}\tilde{\rho}\leq \rho \leq -\frac{1}{2}\tilde{\rho}$. 

For any of the regions I-V, the integral in (\ref{eq-a3}) is estimated as the sum over the four intervals:
\begin{equation}
 \abs{a_3}\leq \left(\int_{K_1}+\int_{K_2}+\int_{K_3}+\int_{K_4} \right)\abs{ {a_1}(z,x; (\sigma, \rho))}\abs{a_2(z,y; (\tilde{\sigma}, \tilde{\rho}-\rho))}d\rho
\end{equation}

We will refer to a particular summand as $R_j$ where $R \in \{I, II, III, IV, V\}$ and $j \in \{K_1,K_2,K_3,K_4\}$. 
For example $\abs{a_3}\leq I_{K_1}+I_{K_2}+I_{K_3}+I_{K_4}$ in region I, $\abs{a_3} \leq II_{K_1}+II_{K_2}+II_{K_3}+II_{K_4}$ in region II, etc. 

On interval 1 (for all five regions), $\frac{1}{3}\abs{\rho}\leq \abs{\tilde{\rho}-\rho} \leq \frac{5}{3}\abs{\rho}$, which means $\langle \tilde\sigma, \tilde{\rho}-\rho \rangle \sim \langle\tilde \sigma, \rho \rangle$. In particular, $\frac{1}{3} \langle \tilde\sigma, \rho \rangle \rangle \leq \langle \tilde\sigma, \tilde{\rho}-\rho \rangle \rangle \leq \frac{5}{3} \langle \tilde\sigma, \rho \rangle \rangle$. Thus:
$R_{K_1} \lesssim \int_{\rho \geq \frac{3}{2} \tilde{\rho}} \langle \sigma, \rho \rangle^M \langle \tilde{\sigma}, \rho \rangle^{\tilde{M}}d\rho$. 

Estimating on the different regions, we get:

$I_{K_1} \lesssim \langle \tilde{\rho} \rangle ^{M + \tilde{M}+1}$ 

$II_{K_1} \lesssim \langle \sigma \rangle ^{M+1}\langle \tilde \sigma \rangle ^{\tilde{M}}$

$III_{K_1}\lesssim \langle \sigma \rangle ^{M}\langle \tilde \sigma \rangle ^{\tilde{M}+1}$

$IV_{K_1}\lesssim \langle \tilde\sigma \rangle ^{\tilde{M}}\langle \tilde{\rho} \rangle ^{M+1}$

$V_{K_1}\lesssim \langle \sigma \rangle ^{M}\langle \tilde \rho \rangle ^{\tilde{M}+1}$

\bigskip
On interval 2 (for all five regions), $\frac{1}{2}\abs{\tilde \rho} \leq \abs{\tilde\rho - \rho}\leq \frac{3}{2}\abs{\tilde{\rho}}$, which means $\abs{\tilde\rho - \rho} \sim \abs{\tilde\rho}$.
Thus, $R_{K_2} \lesssim \langle \tilde \sigma, \tilde\rho \rangle^{\tilde M} \int_{\abs{\rho} \leq \frac{1}{2}\abs{\tilde\rho}} \langle \sigma, \rho \rangle^{M} d\rho$. 

Now we estimate on the different regions to get: 

$I_{K_2} \lesssim \langle \sigma \rangle^{M+1} \langle \tilde\rho \rangle^{\tilde M}$

$II_{K_2}  \lesssim \langle \sigma \rangle^{M}\langle \tilde \sigma \rangle^{\tilde M} \langle \tilde\rho \rangle$

$III_{K_2} \lesssim \langle \sigma \rangle^{M}\langle \tilde \sigma \rangle^{\tilde M} \langle \tilde\rho \rangle$

$IV_{K_2} \lesssim \langle \sigma \rangle^{M+1}\langle \tilde \sigma \rangle^{\tilde M}$

$V_{K_2} \lesssim \langle \sigma \rangle^{M} \langle \tilde\rho \rangle^{\tilde M+1}$

\bigskip
On interval 3 (for all five regions) $\frac{1}{2}\abs{\tilde \rho} \leq \abs{\rho}\leq \frac{3}{2}\abs{\tilde{\rho}}$, which means $\langle \sigma, \rho \rangle \sim \langle \sigma, \tilde\rho \rangle$. Thus, $R_{K_3} \lesssim \langle \sigma, \tilde\rho \rangle^{M} \int_{\abs{\tilde\rho-\rho} \leq \frac{1}{2}\abs{\tilde\rho}} \langle \tilde \sigma, \tilde\rho-\rho \rangle^{\tilde M} d\rho$.

Now we estimate on the different regions to get: 

$I_{K_3} \lesssim \langle \tilde\sigma \rangle^{\tilde M+1} \langle \tilde\rho \rangle^{M}$

$II_{K_3}  \lesssim \langle \sigma \rangle^{M}\langle \tilde \sigma \rangle^{\tilde M} \langle \tilde\rho \rangle$

$III_{K_3} \lesssim \langle \sigma \rangle^{M}\langle \tilde \sigma \rangle^{\tilde M} \langle \tilde\rho \rangle$

$IV_{K_3} \lesssim \langle \tilde \sigma \rangle^{\tilde M}\langle \tilde\rho \rangle^{M+1}$

$V_{K_3} \lesssim \langle \sigma \rangle^{M} \langle \tilde \sigma \rangle^{\tilde M+1} $

\bigskip
On interval 4, both $\abs{\rho} \sim \abs{\tilde{\rho}}$ and $\abs{\tilde\rho-\rho} \sim \abs{\tilde{\rho}}$, while the interval of integration has length comparable to $\tilde{\rho}$. Therefore, 
$R_{K_4} \lesssim \langle \sigma, \tilde \rho \rangle^M \langle \tilde \sigma, \tilde \rho \rangle^{\tilde M}\langle \tilde\rho \rangle.$ 

Now we estimate on the different regions to get: 

$I_{K_4} \lesssim \langle \tilde{\rho} \rangle ^{M + \tilde{M}+1}$ 

$II_{K_4}  \lesssim \langle \sigma \rangle^{M}\langle \tilde \sigma \rangle^{\tilde M} \langle \tilde\rho \rangle$

$III_{K_4} \lesssim \langle \sigma \rangle^{M}\langle \tilde \sigma \rangle^{\tilde M} \langle \tilde\rho \rangle$

$IV_{K_4} \lesssim \langle \tilde \sigma \rangle^{\tilde M}\langle \tilde\rho \rangle^{M+1}$

$V_{K_4} \lesssim \langle \sigma \rangle^{M} \langle \tilde\rho \rangle^{\tilde M+1}$

\bigskip
Finally, we can sum each region across the four intervals, to obtain: 

$\displaystyle  I \lesssim \langle \tilde{\rho}\rangle^{\tilde{M}}\langle \sigma\rangle^{M+1}+\langle \tilde{\rho}\rangle^{M}\langle \tilde\sigma\rangle^{\tilde M+1} $

$\displaystyle  II \lesssim \langle \sigma \rangle^{M+1}\langle \tilde \sigma \rangle^{\tilde M}$ 

$\displaystyle  III \lesssim \langle \sigma \rangle^{M}\langle \tilde \sigma \rangle^{\tilde M+1}$

$\displaystyle  IV \lesssim \langle \sigma \rangle^{M+1}\langle \tilde \sigma \rangle^{\tilde M}$

$\displaystyle  V \lesssim \langle \sigma \rangle^{M}\langle \tilde \sigma \rangle^{\tilde M+1}$.

Regions, II, III, IV and V can be unified into one region as in the lemma, proving the first part of the statement.

Derivatives in $\sigma$ get passed through the integral to the first symbol, the only one which depends on $\sigma$. This effectively reduces $M$ to $M-1$. Thus on $\Gamma_1$, as defined in Lemma \ref{lemma-a3}, we have:
$\abs{\partial_{\sigma} a_3} \leq \langle \sigma \rangle^{-1} \langle \tilde{\rho}\rangle^{\tilde{M}}\langle \sigma\rangle^{M+1}+\langle \tilde{\rho}\rangle^{-1}\langle \tilde{\rho}\rangle^{M}\langle \tilde\sigma\rangle^{\tilde M+1} \leq \langle \sigma \rangle^{-1} \text{(original estimate)}$ since on $\Gamma_1$, $\abs{\tilde{\rho}} \geq \max(\abs{\sigma},\abs{\tilde{\sigma}})$, and therefore $\langle \tilde{\rho}\rangle^{-1} \leq \langle \sigma \rangle^{-1}$. Similarly, derivatives in $\tilde{\sigma}$ get passed under the integral and distributed to  $a_2$, effectively reducing the order of $\tilde{M}$ by 1, and by similar reasoning on $\Gamma_1$, $\abs{\partial_{\tilde \sigma} a_3} \leq \langle\tilde \sigma \rangle^{-1} \text{(original estimate)}.$ On $\Gamma_2$ the same derivative estimates hold automatically, as soon as we reduce the respective exponents by one. 

Now, we turn to the analysis of derivatives with respect to $\tilde\rho$. We use the property $\partial_x(f*g)=\partial_x f*g=f*\partial_x g$. On $\Gamma_1$, which coincides with region I, we distribute the derivative to $a_1$, effectively reducing $M$ by one and estimate the summation of the integral over the four intervals. 
$I \lesssim I_1+I_2+I_3+I_4 \lesssim \langle \tilde{\rho} \rangle ^{M + \tilde{M}}+  \langle \sigma \rangle^{M} \langle \tilde\rho \rangle^{\tilde M}+\langle \tilde\sigma \rangle^{\tilde M+1} \langle \tilde\rho \rangle^{M-1}+ \langle \tilde{\rho} \rangle ^{M + \tilde{M}}.$ There is a gain of $\tilde{\rho}$ everywhere, except on interval 2, so we refine our estimates there by using the mean value property of the derivative and a Taylor expansion to show a similar improvement. In the notation below, we suppress the spatial variables, and denote with $\partial_2$ derivatives with respect to the second phase variable.

\begin{equation*}
  \begin{array}{ll}
       \partial_{\tilde{\rho}}(a_1*a_2)& =(\partial_2 a_1)*a_2 = \int_{\abs{\rho}\leq \frac{1}{2}\abs{\tilde{\rho}}}\partial_2 a_1(\sigma, \rho) a_2(\tilde \sigma, \tilde\rho-\rho)d\rho \\
      &\approx \int_{\abs{\rho}\leq \frac{1}{2}\tilde{\rho}}\partial_2 a_1(\sigma, \rho)(a_2(\tilde \sigma, \tilde\rho) - \rho \partial_2 a_2(\tilde\sigma, \tilde\rho))d\rho\\
      &=\int_{\RR}\partial_2 a_1(\sigma, \rho)a_2(\tilde \sigma, \tilde\rho)d\rho\\
      &-\int_{\abs{\rho}\geq \frac{1}{2}\abs{\tilde{\rho}}}\partial_2 a_1(\sigma, \rho)a_2(\tilde \sigma, \tilde\rho)d\rho -  \partial_2 a_2(\tilde\sigma, \tilde\rho)\int_{\abs{\rho}\leq \frac{1}{2}\abs{\tilde{\rho}}}\rho \partial_2 a_1(\sigma, \rho)d\rho
    \end{array}
 \end{equation*}
By the mean value property, the first integral of the last expression is 0. Putting the absolute value on both sides, we estimate:
\begin{equation*}
  \begin{array}{l}
       \abs{\partial_{\tilde{\rho}}(a_1*a_2)}\\
       =\abs{a_2(\tilde \sigma, \tilde\rho)}\int_{\abs{\rho}\geq \frac{1}{2}\abs{\tilde{\rho}}}\abs{\partial_2 a_1(\sigma, \rho)}d\rho + \abs{\partial_2 a_2(\tilde\sigma, \tilde\rho)}\int_{\abs{\rho}\leq \frac{1}{2}\abs{\tilde{\rho}}}\abs{\rho}\abs{ \partial_2 a_1(\sigma, \rho)}d\rho\\
       \lesssim \langle \tilde \rho \rangle^{\tilde{M}} \int_{\abs{\rho}\geq \frac{1}{2}\abs{\tilde{\rho}}}\langle \rho \rangle^{M-1} d\rho + \langle\tilde{\rho}\rangle^{\tilde{M}-1}(\int_{\abs{\rho}<\abs{\sigma}} + \int_{\abs{\sigma}<\abs{\rho}<\frac{1}{2}\abs{\tilde{\rho}}}) \langle \rho \rangle \langle \sigma, \rho \rangle^{M-1}d\rho\\
       \lesssim \langle \tilde \rho \rangle^{\tilde{M}+M} + \langle\tilde{\rho}\rangle^{\tilde{M}-1}(\int_{\abs{\rho}<\abs{\sigma}} + \int_{\abs{\sigma}<\abs{\rho}<\frac{1}{2}\abs{\tilde{\rho}}}) \langle \rho \rangle \langle \sigma, \rho \rangle^{M-1}d\rho\\
       \lesssim \langle \tilde \rho \rangle^{\tilde{M}+M} + \langle\tilde{\rho}\rangle^{\tilde{M}-1}(\langle \sigma \rangle ^{M-1}\langle \sigma \rangle ^{2}+\langle \sigma \rangle ^{M+1})\\
       \lesssim \langle\tilde{\rho}\rangle^{\tilde{M}-1}\langle \sigma \rangle ^{M+1}
    \end{array}
 \end{equation*}
 
 Now, replacing $I_2$ with this refined estimate, we get: 
 \begin{equation*}
 \begin{split}
 &I \lesssim \langle \tilde{\rho}\rangle^{\tilde{M}-1}\langle \sigma\rangle^{M+1}+\langle \tilde{\rho}\rangle^{M-1}\langle \tilde\sigma\rangle^{\tilde M+1}\\ 
 &\lesssim \langle \tilde{\rho}\rangle^{-1}\text{(original estimate)}\lesssim \langle \sigma, \tilde{\sigma}, \tilde{\rho}\rangle^{-1}\text{(original estimate)},
 \end{split}
 \end{equation*}
 since on $\Gamma_1$, $\langle \sigma, \tilde{\sigma}, \tilde{\rho}\rangle \sim \langle \tilde{\rho}\rangle.$

 $\Gamma_2$ consists of regions II, III, IV, V. In regions III and V, $\langle\sigma\rangle$ dominates, and therefore, $\langle\sigma\rangle \sim \langle\sigma, \tilde\sigma \rangle \sim \langle\sigma, \tilde\sigma, \tilde \rho \rangle.$ We calculate the derivative of the convolution by distributing the derivative to $a_1$, effectively reducing $M$ by 1. We get: 
 \begin{equation*}
 \begin{split}
  \abs{\partial_{\tilde{\rho}}(a_1*a_2)}&\lesssim \frac{\langle \sigma \rangle^{M} \langle \tilde{\sigma} \rangle ^{\tilde{M}+1}}{\langle \sigma, \tilde{\sigma} \rangle} =
 \langle \sigma, \tilde{\sigma} \rangle^{-1}\frac{\langle \sigma \rangle^{M+1} \langle \tilde{\sigma} \rangle ^{\tilde{M}+1}}{\langle \sigma, \tilde{\sigma} \rangle}\\
 &\lesssim \langle \sigma, \tilde{\sigma}, \tilde{\rho} \rangle^{-1}\text{(original estimate)}. 
 \end{split}
 \end{equation*}
 On regions II and IV, where $\langle\tilde\sigma\rangle$ dominates, $\langle\tilde\sigma\rangle \sim \langle\sigma, \tilde\sigma \rangle \sim \langle\sigma, \tilde\sigma, \tilde \rho \rangle.$ Now we distribute the derivative to $a_2$, reducing $\tilde{M}$ by 1. Repeating the analysis above, we obtain the final statement  of the lemma. 

\end{proof}
 
 We continue our analysis of the composition kernel
 \begin{equation}\label{eq-kernel}
  K_{A_1^*A_2}(x,y)=\int_{\RR^{3n+1}}e^{i[(x-z)\cdot \sigma + (z-y)\cdot \tilde{\sigma}+h(z)\tilde{\rho}]} a_3(x,y,z; \sigma, \tilde{\sigma}, \tilde{\rho}) \,d\tilde{\rho}\,d\sigma \,d\tilde{\sigma} \,dz.
 \end{equation}

Let us introduce a smooth function $\chi(t)$, which is identically 0 on $\abs{t}\leq \frac{1}{2}$ and identically 1 on $\abs{t}\geq 1$. We let $\chi_1(z;\sigma, \tilde\sigma, \tilde\rho)=\chi\left(\frac{\abs{\sigma - \tilde\rho \grad h(z)}}{\abs{\sigma, \tilde\rho}}\right)$ and  $\chi_2(z;\sigma, \tilde\sigma, \tilde\rho)=1-\chi_1(z;\sigma, \tilde\sigma, \tilde\rho)$. With this choice, $\{\chi_j\}$ form a partition of unity, subordinate to the sets: 
\begin{equation}\label{eq-good}
\Omega_{good}:=\{(z, \sigma, \tilde\sigma, \tilde\rho) \vert \langle \sigma-\tilde\rho\grad h(z) \rangle > \frac{1}{2}\langle \sigma, \tilde\rho \rangle\}  
\end{equation}
and
\begin{equation}\label{eq-bad}
\Omega_{bad}:=\{(z,\sigma, \tilde\sigma, \tilde\rho) \vert \langle \sigma-\tilde\rho\grad h(z) \rangle < \langle \sigma, \tilde\rho \rangle\}. 
\end{equation}
We will refer to these sets as $\Omega_g$ and $\Omega_b$ for short. Each $\chi_j$ belongs to the standard symbol class $S_{1,0}^0(\RR_z^n \times \RR_{\sigma, \tilde\sigma, \tilde\rho}^{2n+1}\setminus 0)$, which means multiplication by $\chi_j$ does not affect the estimates on $a_3$ and is carried out implicitly in the analysis below.

We start by analyzing the contribution to $K_{A_1^*A_2}$ from $\Omega_g$. We proceed with stationary phase in $z$ and $\tilde\sigma$. The phase function is:
\begin{equation} \label{eq-phase}
 \Phi(x,y,z; \sigma, \tilde\sigma, \tilde\rho)=(x-z)\cdot\sigma+(z-y)\cdot \tilde\sigma+h(z)\tilde\rho
\end{equation} 
Since
$$\Phi_z=-\sigma+\tilde\sigma+\tilde\rho\grad h(z),$$
$$\Phi_{\tilde\sigma}=z-y,$$
the critical point is $z=y$ and $\tilde\sigma=\sigma-\tilde\rho\grad h(z)$.
The Hessian of $\Phi$ is 

$$\cal{H}=\begin{bmatrix} \tilde\rho Hh(z) & I_n \\ I_n & 0 \end{bmatrix},$$
where $Hh(z)$ is the Hessian of the defining function $h$ at $z$. Note that $\abs{\det (\mathcal{H})}=1$.

\begin{equation} \label{eq-sp1}
  K_{A_1^*A_2}(x,y)\sim \int_{\RR^{n+1}}e^{i[(x-y)\cdot \sigma + h(y)\tilde{\rho}]} a_3(x,y,y; \sigma, \sigma-\tilde\rho\grad h(y), \tilde{\rho}) d\tilde{\rho}d\sigma, 
\end{equation}
provided a valid asymptotic expansion, which we justify below. 
Now, let $b(x,y; \sigma, \tilde\rho)=a_3(x,y,y; \sigma, \sigma-\tilde\rho\grad h(y), \tilde{\rho})$.
$b$ is supported on $\Omega_g$, and satisfies the estimate
$$\abs{\partial_{\sigma}^{\gamma} \partial_{\tilde\rho}^{\alpha}b} \lesssim \langle \sigma \rangle^{M+1-\abs{\gamma}}\an{\sigma, \tilde\rho}^{\tilde M-\abs{\alpha}}$$

This means that $b\in S^{\tilde M, M+1}(\RR^{2n}; \RR^1\setminus\{0\}, \RR^n)$, a symbol-valued symbol as defined by (\ref{eq-smm'}). 

Now, we analyze
\begin{equation} \label{eq-b1}
 u_g=\int_{\RR^{n+1}}e^{i[h(y)\tilde{\rho}+(x-y)\cdot \sigma]} b(x,y; \sigma, \tilde{\rho})\,d\tilde{\rho}\,d\sigma.
\end{equation}

By \cite{Mendoza82}, we can analyze the phase function in two parts. Let $\phi_1(x,y; \tilde\rho, \sigma)=h(y)\tilde\rho+(x-y)\cdot\sigma$, and $\phi_0=\phi_1(x,y; \tilde\rho, 0)=h(y)\tilde\rho$.
$\phi_0$ defines the $2n$-dimensional Lagrangian 
$$\Lambda_0=\{(x, y, 0, \tilde\rho \grad h(y)) \vert x\in \RR^n, y \in S, \tilde\rho \in \RR\setminus 0 \}$$ and the associated canonical relation:
$$ C_0:= \Lambda_0'=\{(x, 0, y, -\tilde\rho \grad h(y)) \vert x\in \RR^n, y \in S, \tilde\rho \in \RR\setminus\ 0\} = \Lambda_0. $$
$\phi_1$ defines the $2n$-dimensional Lagrangian 
$$\Lambda_1=\{(x, x, \sigma, -\sigma+\tilde\rho \grad h(y)) \vert x\in S, \sigma \in \RR^n, \tilde\rho \in \RR, (\sigma, \rho) \neq ({\bf 0},0)\}$$ 
and the associate canonical relation:
$$ C_1:= \Lambda_1'=\{(x, \sigma, x, \sigma-\tilde\rho \grad h(x)) \vert x\in S, \sigma \in \RR^n, \tilde\rho \in \RR, (\sigma, \rho) \neq({\bf 0},0)\}=C_{\Sigma}$$
The two canonical relations intersect cleanly in codimension $n$. Therefore, by Mendoza, $u_g \in I^{p,l}(C_0', C_{\Sigma}')$. To calculate the subscripts, we note that $p=\text{order }u_g \vert_{(C_{\Sigma}\setminus C_0)}$ and $\sigma_{pr}(u_g \vert_{(C_{\Sigma}\setminus C_0)})\sim \text{dist}((x, \sigma), C_0\cap C_{\Sigma})^{-l-\frac{n}{2}}$.
Therefore $p=M+\tilde M +1 + \frac{n+1}{2}-\frac{2n}{4}=M+\tilde M+\frac{3}{2}=m_1-\frac{1}{2}+m_2-\frac{1}{2}+\frac{3}{2}=m_1+m_2+\frac{1}{2}$. For the second subscript, we know $-l-\frac{n}{2}=M+1=m_1+\frac{1}{2}$, so $l=-\frac{n+1}{2}-m_1$. Thus, $u_g \in I^{m_1+m_2+\frac{1}{2}, -\frac{n+1}{2}-m_1}(C_0', C_{\Sigma}')$, and therefore so does $K_{A_1^*A_2}$ on $\Omega_g$, as long as we can justify the asymptotic expansion, implied by (\ref{eq-sp1}).

In order to justify the expansion, we need to show that applying the operator $L:=\div{\cal{H}^{\text{-1}} \grad}$
to $a_3$ and evaluating at the critical point leads to additional decay in the elliptic variable, and therefore in $\an{\sigma, \tilde\rho}$. Here,
$$\cal{H}^{\text{-1}}=\begin{bmatrix} 0 & I_n \\ I_n & -\tilde\rho Hh(z) \end{bmatrix},$$

We can rewrite $L$ as
$$L=\sum_{i=1}^n 2 \frac{\partial^2}{\partial z_i \partial \tilde{\sigma}_i}-\tilde \rho \sum_{i,j=1}^n \frac{\pa^2 h}{\pa z_i \pa z_j} \frac{\pa^2}{\pa \tilde\sigma_i \pa \tilde\sigma_j}:=L_1+L_2.$$

Applying $L_1$ to $a_3$ results in a gain of $\an{\tilde\sigma}^{-1}$, which at the critical point ($z=y, \tilde\sigma=\sigma - \tilde\rho \grad h(z)$) is a gain of $\an{\sigma, \tilde\rho}^{-1}$ on $\Omega_g$. Similarly, $\frac{\pa^2}{\pa \tilde\sigma_i \pa \tilde\sigma_j}$ leads to a gain of $\an{\tilde\sigma}^{-2}\sim\an{\sigma,\tilde\rho}^{-2}$, so the multiplication by $\tilde\rho$ in $L_2$ leads to a gain of $\an{\sigma,\tilde\rho}^{-1}$ as claimed.

It remains to understand $K_{A_1^*A_2}$ on the complement of $\Omega_g$, namely $\Omega_b$, as defined in (\ref{eq-bad}).
On $\Omega_b$, $\an{\sigma}\sim\an{\tilde\rho}$. Recall the space partition into regions $\Gamma_1$ and $\Gamma_2$, defined in lemma \ref{lemma-a3}. Specializing these to $\Omega_b$, we define:
$$\Gamma_{b,1}:=\{\abs{\sigma}\sim \abs{\tilde\rho} \geq c_1 \abs{\tilde{\sigma}}\},$$
$$\Gamma_{b,2}:=\{\abs{\sigma}\sim \abs{\tilde\rho} \leq c_2 \abs{\tilde{\sigma}}\}.$$ 

We proceed to prove that the contribution to $K_{A_1^*A_2}$ from i) $\Gamma_{b,1}$ is in $I^{p,l}(C_0^{t'}, C_{\Omega}')$, where  $p$ is as calculated before and $l=-\frac{n+1}{2}-m_2$, and $C_0^t$ is the transpose of $C_0$, and ii) $\Gamma_{b,2}$ is $C^{\infty}$. 

To prove i) we note that on $\Gamma_{b,1}$, $(\tilde{\rho},\sigma)$ are the larger variables, which dominate $\tilde\sigma$. Thus, we apply stationary phase in $z, \sigma$, instead of $z, \tilde\sigma$.

The critical point is now $z=x$ and $\sigma=\tilde\sigma+\tilde\rho\grad h(z)$. 

$$\cal{H}=\begin{bmatrix} \tilde\rho Hh(z) & -I_n \\ -I_n & 0 \end{bmatrix}, \cal{H}^{\text{-1}}=\begin{bmatrix} 0 & -I_n \\ -I_n & -\tilde\rho Hh(z) \end{bmatrix},$$
with $\abs{\det (\mathcal{H})}=1$. The operator in the stationary phase expression is now 
$$L=\sum_{i=1}^n -2 \frac{\partial^2}{\partial z_i \partial {\sigma}_i}-\tilde \rho\sum_{i,j=1}^n \frac{\pa^2 h}{\pa z_i \pa z_j} \frac{\pa^2}{\pa \sigma_i \pa \sigma_j}:=L_1+L_2,$$
leading to a gain of $\an{\sigma}^{-1}\sim\an{\sigma, \tilde\rho}^{-1}$ in $a_3$, which evaluated at the critical point is a gain of $\an{\tilde\rho}^{-1}\sim\an{\tilde\rho, \tilde\sigma}^{-1}$. Thus, on $\Gamma_{b,1}$,
\begin{equation}\label{eq-b2}
  u_b=\int_{\RR^{n+1}}e^{i[h(x)\tilde{\rho}+(x-y)\cdot \tilde\sigma]} b(x,y; \tilde\sigma, \tilde{\rho}) \,d\tilde{\rho}\,d\tilde\sigma
\end{equation}

Here  $b(x,y; \sigma, \tilde\rho)=a_3(x,y,x;  \tilde\sigma+\tilde\rho\grad h(y), \tilde\sigma, \tilde{\rho})$ and $b \in S^{M, \tilde M+1}(\RR^{2n}; \RR^1\setminus 0, \RR^n)$. Notice the similarity between $u_g$ and $u_b$ (equations (\ref{eq-b1}) and (\ref{eq-b2})), except $h(y)$ in $u_g$ is replaced by $h(x)$ in $u_b$. Similarly to before, we define:
$\phi_1(x,y; \tilde\rho, \tilde\sigma)=h(x)\tilde\rho+(x-y)\cdot\tilde\sigma$ and $\phi_0=\phi_1(x,y; \tilde\rho, 0)=h(x)\tilde\rho$. Then:

$$ \Lambda_0'=\{(x, \tilde\rho \grad h(x), y, 0) \vert x\in S, y \in \RR^n, \tilde\rho \in \RR\setminus 0\}=: C_0^t,$$
$$ \Lambda_1'=\{(x, \tilde\sigma+\tilde\rho \grad h(x), x, \tilde\sigma) \vert x\in S, \tilde\sigma \in \RR^n, \tilde\rho \in \RR, (\tilde\sigma, \tilde\rho) \neq(0^n,0)\}=C_{\Sigma}.$$

Once again, $C_0^t$ intersects $C_{\Sigma}$ cleanly in codimension $n$
$$p=M+\tilde M +1 + \frac{n+1}{2}-\frac{2n}{4}=m_1+m_2+\frac{1}{2}$$
$$-l-\frac{n}{2}=\tilde M+1=m_2+\frac{1}{2}$$
$$l=-\frac{n+1}{2}-m_2$$

Thus,  $u_b \in I^{m_1+m_2+\frac{1}{2}, -\frac{n+1}{2}-m_2}(C_0^{t'}, C_{\Sigma}')$, and therefore so does the contribution to $K_{A_1^*A_2}$ from $\Gamma_{b,1}$.
  
Last, we analyze the properties of $K_{A_1^*A_2}$ on $\Gamma_{b,2}$. For the phase function in (\ref{eq-phase}), we have
$$\abs{\grad_z \Phi}=\abs{-\sigma+\tilde\sigma+\tilde\rho \grad h(z)}\geq \eps \abs{\tilde\sigma}.$$
Since the gradient never vanishes, there exists a differential operator
$$L_z=\sum a_j(z, \tilde\sigma) \frac{\pa}{\pa z_j}, a_j\in S^{-1}, L_z(e^{i\Phi})=e^{i\Phi}$$

We can integrate by parts as many times, $N$, as we want and 
\begin{equation}
(L_z^t)^N(a_3)\lesssim \an{\sigma, \tilde{\rho}}^{M+1}\an{\tilde{\sigma}}^{\tilde M-N}. 
\end{equation}
Using the fact that on $\Gamma_{b,2}$, $\abs{\sigma, \tilde\rho}\leq c\abs{\tilde\sigma}$, we obtain:
$$K_{A_1^*A_2}(x,y)\lesssim \int_{\RR^n} \int_{\RR^n} \int_{\abs{\sigma, \tilde\rho}\leq c\abs{\tilde\sigma}}\an{\sigma, \tilde{\rho}}^{M+1}\an{\tilde{\sigma}}^{\tilde M-N} d\sigma d\tilde\rho d\tilde\sigma dz$$
which converges absolutely, provided $N$ is large enough. Increasing the number of times, $N$, in which we differentiate by parts, we can see that $K_{A_1^*A_2}\in C^{\infty}(\Gamma_{b,2})$.

\bibliographystyle{plain} 
\bibliography{GreensFunction} 

\end{document}